\documentclass[reqno]{amsart}
\usepackage[margin = 1in]{geometry}
\usepackage{amsmath, amssymb, amsthm, fancyhdr, verbatim, graphicx}
\usepackage{enumerate}
\usepackage[all]{xy}
\usepackage[dvipsnames]{xcolor}
\usepackage{mathrsfs}
\usepackage{tikz-cd}
\usepackage{framed}
\usepackage[hidelinks, pagebackref]{hyperref}
\hypersetup{colorlinks=true, linkcolor=magenta,citecolor=cyan} 
\usepackage[titletoc]{appendix}
\usepackage{bbm}
\usepackage{lipsum}
\usepackage{adjustbox}
\usepackage{commands}
\usepackage{envi}
\usepackage{stmaryrd}
\usepackage{leftindex}
\usepackage[english]{babel}
\usepackage{dynkin-diagrams}

\numberwithin{equation}{section}

\makeatletter
\def\th@remark{%
  \thm@headfont{\bfseries}%
  \normalfont 
  \thm@preskip \thm@preskip 
  \thm@postskip\thm@preskip
}
\def\imod#1{\allowbreak\mkern5mu({\operator@font mod}\,\,#1)}
\makeatother

\numberwithin{equation}{section}

\title{Arithmetic volume of Shtukas and Langlands duality}
\author{Zeyu Wang}
\address{Massachusetts Institute of Technology, Department of Mathematics, 77 Massachusetts Avenue, Cambridge, MA 02139, USA}
\email{wangzeyu@mit.edu}

\author{Wenqing Wei}
\address{University of California Berkeley, Department of Mathematics, Berkeley, CA 94720, USA}
\email{weiwenqing@berkeley.edu}

\begin{document}

\begin{abstract}
    We extend the work of Feng--Yun--Zhang relating the arithmetic volume of Shtukas to derivatives of zeta functions by allowing arbitrary coweights for split semisimple algebraic groups. As in their original work, the formula involves some numbers called eigenweights. We obtain uniform formulas for the eigenweights in terms of the Langlands dual group, marking the first structural role for the dual group in such formulas governing derivatives of $L$-functions.
\end{abstract}

\maketitle
\tableofcontents

\section{Introduction}
\subsection{Motivation}
Shimura varieties play a fundamental role in the Langlands correspondence over number fields, and their geometry encodes deep arithmetic information. One manifestation of this philosophy is that the volume of Shimura varieties is related to Dirichlet $L$-functions of number fields, while their arithmetic volume, namely the volume of integral models, is related to the first derivative of these $L$-functions. See \cite[\S1.2]{FYZvolume} for a summary of the literature.

The moduli of Shtukas provide the function field analogue of (integral models of) Shimura varieties. In \cite{FYZvolume}, the arithmetic volume of moduli stacks of Shtukas with minuscule modification type was studied and related to higher derivatives of the zeta function of the curve. Compared with the number field setting, two striking new features arise. First, derivatives of arbitrary order naturally appear, in contrast to the number field case where only the first derivative is expected. Second, certain nontrivial constants, referred to as \emph{eigenweights}, enter the formulas. While their meaning remains somewhat mysterious in the number field setting, they admit a natural geometric interpretation in the function field context. This additional structure can in turn be used to predict conjectural formulas for arithmetic volumes of Shimura varieties, providing part of the motivation for the work of \cite{FYZvolume}.

However, the results of \cite{FYZvolume} are restricted to minuscule modification types, which obscures the conceptual unity of the theory. Moreover, although eigenweights admit a geometric definition, their explicit computation remains intricate. Even for groups of type $A$, the formulas obtained in \cite{feng2026eigenweightsarithmetichirzebruchproportionality} involve complicated information such as the character tables of symmetric groups.

In this article, we extend these results to arbitrary modification types and provide a conceptual and uniform description of eigenweights in terms of the Langlands dual group. To the best of our knowledge, this is the first instance in which the Langlands dual group plays a direct and structural role in formulas governing derivatives of $L$-functions. In addition, we carry out explicit computations of eigenweights in most cases of fundamental importance, leading to formulas that are more elementary than those previously known.

\subsection{Main result}
We now present the main result of this article and recall the necessary background.
\subsubsection{Moduli stack of Shtukas}
Fix a proper smooth geometrically connected curve $C$ defined over a finite field $\FF_q$. Let $G$ be a split semisimple algebraic group defined over $\FF_q$, and let $\Bun_G$ be the moduli stack of principal $G$-bundles on $C$. Fix a maximal torus $T\subset G$. For each dominant coweight $\lambda\in X_*(T)_+$, one has the Hecke stack
\[
\begin{tikzcd}
    \Bun_G & \ar[l, "\overline{\lh}"'] \Hk_{G,\leq \lambda} \ar[r, "\overline{\rh}"] & \Bun_G
\end{tikzcd}.
\] Here, we use $\overline{\lh}$ (resp. $\overline{\rh}$) to denote the map remembering only the $G$-bundle on the left (resp. the right). We put an overline to distinguish it from the Hecke map that remembers also the point on the curve.

The \emph{moduli stack of Shtukas with one leg} $\Sht_{G,\leq\lambda}$ is defined via the Cartesian square
\[
\begin{tikzcd}
    \Sht_{G,\leq\lambda} \ar[r, "f_{\Sht}"] \ar[d] & \Hk_{G,\leq \lambda} \ar[d, "{(\overline{\lh},\Frob\circ \overline{\rh})}"] \\
    \Bun_G \ar[r, "\D_{\Bun_G}"] & \Bun_G\times\Bun_G
\end{tikzcd}
\]
where $\Frob:\Bun_G\to\Bun_G$ is the Frobenius map and $\D_{\Bun_G}:\Bun_G\to\Bun_G\times\Bun_G$ is the diagonal map.

Moreover, for a sequence of dominant coweights $\lambda_I=(\lambda_1,\cdots,\lambda_r)\in X_*(T)^r_+$ where $I=\{1,\cdots,r\}$, one can define the iterated Hecke stack
\[
\Hk_{G,\leq\lambda_I}:=\Hk_{G,\leq \lambda_1}\times_{\Bun_G}\Hk_{G,\leq \lambda_2}\times_{\Bun_G}\cdots\times_{\Bun_G}\Hk_{G,\leq \lambda_r}.
\]

It defines a correspondence
\[
\begin{tikzcd}
    \Bun_G & \ar[l, "\overline{\lh}_I"'] \Hk_{G,\leq \lambda_I} \ar[r, "\overline{\rh}_I"] & \Bun_G
\end{tikzcd}.
\]

One defines the \emph{moduli stack of (iterated) Shtukas with $r$-legs} via the Cartesian square
\[
\begin{tikzcd}
    \Sht_{G,\leq\lambda_I} \ar[r, "f_{\Sht,I}"] \ar[d] & \Hk_{G,\leq \lambda_I} \ar[d, "{(\overline{\lh}_I,\Frob\circ \overline{\rh}_I)}"] \\
    \Bun_G \ar[r, "\D_{\Bun_G}"] & \Bun_G\times\Bun_G
\end{tikzcd}.
\]

\subsubsection{Arithmetic volume of the moduli of Shtukas}

Let $d=\dim \Sht_{G,\leq\lambda_I}$. For a top-degree cohomology class $\alpha\in H^{2d}(\Sht_{G,\leq\lambda_I})$, viewed as a ``top form,'' it is natural to consider the corresponding ``volume'' of $\Sht_{G,\leq\lambda_I}$. However, this notion is not well-behaved, since $\Sht_{G,\leq\lambda_I}$ is neither proper nor smooth. In fact, one has $H^{2d}(\Sht_{G,\leq\lambda_I})=0$ whenever $r>0$ and $\lambda_I\neq (0,\ldots,0)$.

In \cite{FYZvolume}, the authors introduced an \emph{ad hoc} definition of this volume in the special case where $\alpha$ can be thought of as $f_{\Sht,I}^*\alpha_0$ for some $\alpha_0\in H^{2d}(\Hk_{G,\leq\lambda_I})$. The idea originates from the Grothendieck--Lefschetz trace formula. More precisely, one considers the maps
\[
H^*_c(\Bun_G)\xrightarrow{(\overline{\rh}_I)^*} IH_c^*(\Hk_{G,\leq\lambda_I})
\xrightarrow{(\overline{\lh}_I)_!} H^{*-2d}_c(\Bun_G),
\]
where $IH_c^*(\Hk_{G,\leq\lambda_I})$ denotes the compactly supported intersection cohomology of $\Hk_{G,\leq\lambda_I}$.\footnote{In \cite{FYZvolume}, only the case where $\Hk_{G,\leq\lambda_I}$ is smooth is considered. The extension to the non-smooth case via intersection cohomology is straightforward.} We refer to \S\ref{sec:detlbloc} for a precise definition of these maps.

Given the data above, one considers the operator \begin{equation}\label{eq:gammaintro}\Gamma_{c,\alpha_0}:H^*_c(\Bun_G)\to H^*_c(\Bun_G)\end{equation} defined by \[\Gamma_{c,\alpha_0}(-):=(\overline{\lh}_I)_!((\overline{\rh}_I)^*(-)\cup \alpha_0).\] Note that the operator $\Gamma_{c,\alpha_0}$ preserves the cohomological degree. One defines the \emph{arithmetic volume} of the moduli space $\Sht_{G,\leq\lambda_I}$ with respect to ``$\alpha$" to be the number \footnote{Here, the trace is understood as the alternating sum of the traces on cohomology groups.} \begin{equation}\label{eq:volumedefintro}
    \mathrm{vol}(\Sht_{G,\leq\lambda_I},\alpha):= \tr(\Frob\circ \Gamma_{c,\alpha_0},H^*_c(\Bun_G)).
\end{equation} 

See Remark \ref{rmk:converg} for a discussion of convergence issues related to this definition.

\begin{example}
When $r=0$, one has $\Hk_{G,\varnothing}=\Bun_G$ and $\Sht_{G,\varnothing}^0=\Bun_G(\FF_q)$. We choose $\alpha_0=1\in H^0(\Bun_G)$. In this case, we have $\mathrm{vol}(\Bun_G(\FF_q),1)=\tr(\Frob,H_c^*(\Bun_G))$. Modulo convergence issues, the Grothendieck--Lefschetz fixed point formula predicts that $\mathrm{vol}(\Bun_G(\FF_q),1)=|\Bun_G(\FF_q)|$. This matches the natural expectation that the volume of the discrete stack $\Bun_G(\FF_q)$ is its size.
\end{example}

\subsubsection{Main result: same modification type}\label{sec:mainresultsame}
Now we specialize to a very canonical choice of $\alpha_0$. When $r=1$, there is a distinguished line bundle $\cL_{\det}\in \Pic(\Hk_{G})_{\QQ}$ called the \emph{determinant line bundle}, obtained by pulling back its local counterpart defined in \S\ref{sec:detlbloc}.

For general $r$, consider the natural projections $r_i:\Hk_{G,\leq \lambda_I}\to \Hk_{G,\leq\lambda_i}$ for $i\in I$. Define
\[
\cL_{\det,I}:=\bigl((r_i)^*\cL_{\det}\bigr)_{i\in I}\in \Pic(\Hk_{G,\leq\lambda_I})_{\QQ}^I.
\]
We take
\[
\alpha_0=\prod_{i=1}^{r} c_1\bigl((r_i)^*\cL_{\det}\bigr)^{d_i}\in H^{2d}(\Hk_{G,\leq\lambda_I}),
\quad \text{where } d_i=\langle 2\rho,\lambda_i\rangle +1.
\]
In this case, we write
\begin{equation}\label{eq:gammaintrodet}\Gamma_{c,\cL_{\det,I}}=\Gamma_{c,\alpha_0}\end{equation}
\[
\mathrm{vol}(\Sht_{G,\leq\lambda_I},f_{\Sht,I}^*\cL_{\det,I})
=
\mathrm{vol}(\Sht_{G,\leq\lambda_I},\alpha).
\]

We first state a version of the main result when all $\{\lambda_i,1\leq i\leq r\}$ are the same, which already reflects the most interesting aspects of this problem. See \S\ref{sec:mainresultdiff} for a more general version.

\begin{thm}\label{thm:mainsame}
    For $\lambda_I=(\lambda,\lambda,\cdots,\lambda)$ where $\lambda\in X_*(T)_+$, assuming $\Sht_{G,\lambda_I}\neq \varnothing$, we have
    \begin{equation}
        \mathrm{vol}(\Sht_{G,\leq\lambda_I},f_{\Sht,I}^*\cL_{\det,I})
        =
        |\pi_1(G)|q^{\dim \Bun_G}(\log q)^{-r}
        \Bigl(\dfrac{d}{ds}\Bigr)^r\Big|_{s=0}
        \Bigl(q^{(2g-2)b_{\lambda}s}\displaystyle\prod_{i=1}^n\zeta_C(\epsilon_{\lambda,i}s+d_i)\Bigr).
    \end{equation}
    Here, \begin{itemize}
        \item $n=\mathrm{rk}(G)$ is the rank of the group $G$,
        \item $d_i, 1\leq i\leq n$ are the degrees of fundamental invariants of $G$, that is, the degree of generators of the free polynomial ring $\Qlbar[\mathfrak{g}]^G$ where $\frg$ is the Lie algebra of $G$.
        \item $b_{\lambda}\in \QQ,\epsilon_{\lambda,i}\in \overline{\QQ}$ are some constants,
        \item $\zeta_C(s)$ is the zeta function of $C$.
        \item $g$ is the genus of the curve $C$.
    \end{itemize}
\end{thm}

    When $\lambda$ is minuscule, Theorem \ref{thm:mainsame} is a special case of \cite[Theorem\,1.3.8, 5.6.9]{FYZvolume}. See \S\ref{sec: comparision with FYZ} and Remark \ref{rmk:compareFYZlb} for details.

    The numbers $\epsilon_{\lambda,i}$ are called \emph{eigenweights} in \cite{FYZvolume}.\footnote{Our numbers $\epsilon_{\lambda,i}$ differ from those in \cite{FYZvolume} by signs. Indeed, we use a normalization in which the eigenweights are expected to be positive.} They are rational numbers when $G$ is almost simple and not of type $D_n$ for $n\geq 4$ even. See Example \ref{eg:Edet} for the precise meaning of these numbers. In \S\ref{sec:Eexamples}, we explicitly calculate these numbers for all coweights of classical groups and some fundamental coweights for exceptional groups, resulting in formulas that are more elementary than those in \cite{feng2026eigenweightsarithmetichirzebruchproportionality}.

    The constants $b_{\lambda}$ are less emphasized in \cite{FYZvolume} as they are easy when $\lambda$ is minuscule. 
    See \S\ref{sec:bexamples} for some computational results on these numbers.

    In \S\ref{sec:langlandsdual}, we will give a uniform description of the constants $b_{\lambda},\epsilon_{\lambda,i}$ in terms of the Langlands dual group $\Gc$ of $G$.

\begin{remark}
    In a separate AI-assisted work \cite{eigenweights}, we give a closed formula for the eigenweights $\epsilon_{\lambda,i}$ and prove that
    \[
    b_{\lambda}=\frac{1}{2}\sum_{i=1}^n\epsilon_{\lambda,i}.
    \]
\end{remark}

\begin{remark}
    Generalization of Theorem \ref{thm:mainsame} to the case where $G$ is split reductive should be straightforward. We only state the result for semisimple groups to keep the formula clean and compact.
\end{remark}

\begin{remark}
    Theorem \ref{thm:mainsame} can be regarded as an instance of relative Langlands duality in the sense of \cite{BZSV}: On the automorphic side, one takes the trivial spherical $G$-variety $G \curvearrowright \mathrm{point}$ corresponding to the constant period; on the spectral side, one takes the twisted cotangent bundle $\Gc \curvearrowright T^*_{\psi}(\Gc/\Nc)$ corresponding to the Whittaker period. From this point of view, the constants $b_{\lambda},\epsilon_{\lambda,i}$ can be read off from the Poisson structure on the local Plancherel algebra introduced in \cite[\S8]{BZSV} (see Remark \ref{rmk:BZSV}). 
    
\end{remark}

\subsubsection{Main result: different modification types}\label{sec:mainresultdiff}
Now we give a more general version of Theorem \ref{thm:mainsame} in the case where $\lambda_I=(\lambda_1,\cdots,\lambda_r)\in X_*(T)^r_+$ and the $\lambda_i$ are not necessarily equal.

For each $i\in I$, define a differential operator on $\mathbb{R}^n=\{(s_1,\cdots,s_n)\mid s_j\in \mathbb{R}\}$ by
\begin{equation}\label{eq:diffintro}
D_{\lambda_i}:=(2g-2)b_{\lambda_i}+(\log q)^{-1}\sum_{j=1}^n\epsilon_{\lambda_i,j}\partial_{s_j}.
\end{equation}
Here the numbers $b_{\lambda_i}$ and $\epsilon_{\lambda_i,j}$ are the same as those in Theorem \ref{thm:mainsame} for the coweight $\lambda_i$. See \S\ref{sec:eigenweights} for the ordering of these numbers.

Consider the $L$-function
\begin{equation}
\mathscr{L}_{C,G}(s_1,\cdots,s_n):=\prod_{i=1}^n\zeta_C(s_i+d_i).
\end{equation}

The following theorem is a special case of Theorem \ref{thm:main}.

\begin{thm}\label{thm:maindiff}
Under Assumption \ref{assumption:paircomm} (which always holds when $G$ is almost simple and is not of type $D_n$ with $n$ even) and assuming $\Sht_{G,\leq\lambda_I}\neq\varnothing$, we have
\begin{equation}
\mathrm{vol}(\Sht_{G,\leq\lambda_I},f_{\Sht,I}^*\cL_{\det,I})=
|\pi_1(G)|q^{\dim\Bun_G}\big((\prod_{i=1}^r D_{\lambda_i})\mathscr{L}_{C,G}(s_1,\cdots,s_n)\big)\big|_{s_1=\cdots=s_n=0}.
\end{equation}
\end{thm}

\subsection{New ingredients}\label{sec:whatsnew}
Theorem \ref{thm:maindiff} generalizes \cite[Theorem\,1.3.8]{FYZvolume} beyond the minuscule case. As is clear from the formulation, the main task in proving the theorem is to compute the operator $\Gamma_{c,\cL_{\det,I}}$ in \eqref{eq:gammaintrodet}, which we call \emph{(global) relative Hecke operators}. In \cite{FYZvolume}, the computation of relative Hecke operators relies on the Vinberg semigroup and the wonderful compactification of the adjoint group, which are only applicable in the minuscule case. 

Instead, we develop a new approach to compute relative Hecke operators: we study their local counterparts \eqref{eq:localhmap}, which we call \emph{local relative Hecke operators}, and then compute the global ones via a local--global compatibility result (Theorem \ref{thm:universalcommutator}). This approach has several advantages. First, it allows us to handle the non-minuscule case, and indeed all modification types in a uniform way. Second, our local models are directly related to the Satake category, which enables a uniform description of the eigenweights in terms of the Langlands dual group (achieved in \S\ref{sec:langlandsdual}). Finally, our approach is closely connected to the relative Langlands duality in the sense of \cite{BZSV}, and is directly comparable with the approach in \cite{liu2025higherperiodintegralsderivatives} which deals with the strongly tempered case. This perspective suggests a unified conceptual framework for understanding Gross--Zagier type formulas over function fields, namely, formulas in which higher derivatives of $L$-functions arise. See Remark \ref{rmk:BZSV} for further discussion of this connection.





\subsection{Notations}\label{sec:notations}
We use $C$ to denote a smooth geometrically connected projective curve defined over $\FF_q$. Let $\xi\in H^2(C)$ be the fundamental class of $C$. Let $D=\Spec \FF_q[\![t]\!]$ be the formal disc, and $\Aut(D)$ be the group of origin fixing automorphisms of $D$.

Throughout the paper, we work with a split semisimple connected algebraic group $G$. Fix a maximal torus and Borel subgroup $T\sub B\sub G$. We use $W$ to denote the Weyl group of $G$.

We use $L^+G, LG$ to denote the jet, loop group of $G$. We use $\Gr_G:=LG/L^+G$ to denote the affine Grassmannian, which is an ind-scheme. We define the local Hecke stack $\Hk_G^{\loc}:=(L^+G\bs \Gr_G)/\Aut(D)$. As a global counterpart, we have the global Hecke stack $\Hk_G$ which admits a map $\rh:\Hk_G\to \Bun_G\times C$ whose fibers are isomorphic to $\Gr_G$.

For each $\lambda\in X_*(T)$, we have the closed Schubert cells $\Gr_{G,\leq\lambda}\sub \Gr_G$, $\Hk_{G,\leq \lambda}^{\loc}\sub \Hk_{G}^{\loc}$, $\Hk_{G,\leq\lambda}\sub \Hk_{G}$ defined such that $\Gr_{G,\leq\lambda}=\Gr_{G,\leq\lambda^+}$ where $\lambda^+\in X_*(T)_+$ is the unique dominant element in $W\cdot \lambda\sub X_*(T)$.

We fix a prime $\ell\neq \operatorname{char}(\FF_q)$ and take $k=\Qlbar$. For a prestack $X$ over $\FF_q$, we use $\Shv(X)$ to denote the category of (ind-)constructible \'etale sheaves on $X$ with coefficients in $k$. We refer to \cite[\S4.1]{liu2025higherperiodintegralsderivatives} for a detailed definition of this category. We use $\uk_X\in \Shv(X)$ to denote the constant sheaf on $X$, and $\omega_X=(X\to *)^! k\in \Shv(X)$ to denote the dualizing sheaf on $X$ whenever the $!$-pullback functor is defined.  For a prestack $X$, we have its cochain complex $\Gamma(X):=\Gamma_{\textup{\'et}}(X_{\overline{\FF_q}},\uk)$. Define the cohomology ring $H^*(X):=\bigoplus_{i \in \ZZ} H^i\Gamma(X)$. We use $\langle n\rangle = [n](n/2)$ to denote the Tate twist by $n/2$ and cohomological shift by $n$.

Although this is not essential for our purposes, we work by default with $(\infty,1)$-categories. For a category $\mathcal{C}$ and objects $x,y \in \mathcal{C}$, we denote by $\Hom(x,y)$ the mapping space from $x$ to $y$, and write $\Hom^0(x,y) := \pi_0\big(\Hom(x,y)\big)$.

\subsection*{Acknowledgment}
Z. Wang would like to thank his advisor Zhiwei Yun for introducing this problem, sharing his insights, and providing constant support and encouragement. He is also grateful to Wei Zhang for valuable discussions related to this work. W. Wei would like to thank his advisor Tony Feng for introducing this problem and providing constant support. The authors would also like to thank the anonymous referee for their careful reading and helpful comments.

\section{General formalism}\label{sec:generalformalism}
In this section, we set up a general framework for defining and computing the operator in \eqref{eq:gammaintrodet}, which we call \emph{(global) relative Hecke operators}. Namely, whenever the operator has a local origin (coming from a local volume datum as defined in Definition \ref{def:localvolumedatum}), we are able to compute it. The main result of this section is Theorem \ref{thm:main}, which generalizes Theorem \ref{thm:mainsame} and Theorem \ref{thm:maindiff}.

We refer to \S\ref{sec:notations} for basic notations.

\subsection{Volume data}
We first introduce the general setup in which we can define and compute the relative Hecke operators. To summarize, whenever one has some local volume data (to be introduced in \S\ref{sec:localvoldata}), we will be able to define some relative Hecke operators (to be defined in \S\ref{sec:globalrelhecke}).

\subsubsection{Local volume data}\label{sec:localvoldata}
The local Hecke stack $\Hk_G^{\loc}$ introduced in \S\ref{sec:notations} fits into a correspondence \begin{equation}\label{diag:localhecke}\begin{tikzcd}
    {\mathbb{B}(L^+G\rtimes\Aut(D))}  & \Hk_{G}^{\loc} \ar[r, "\rh_{\loc}"] \ar[l, "\lh_{\loc}"'] & {\mathbb{B}(L^+G\rtimes\Aut(D))}
\end{tikzcd}\end{equation} where the two maps have clear meaning from the description $\Hk_G^{\loc}=(L^+G\rtimes\Aut(D))\backslash (LG\rtimes\Aut(D))/(L^+G\rtimes\Aut(D))$.

\begin{defn}\label{def:localvolumedatum}
    We define a \emph{local volume datum} to be a triple $v=(\cK,\frc^{\loc},\frd^{\loc})$ in which: \begin{itemize}
        \item $\cK\in\Shv(\Hk_G^{\loc})$ is a sheaf supported on $\Hk_{G,\leq\lambda}^{\loc}\sub \Hk_G^{\loc}$ for some $\lambda\in X_*(T)_+$.
        \item $\frc^{\loc}\in\Cor_{\Hk_{G}^{\loc},\cK\langle-d_{\frc}\rangle}(\uk_{\mathbb{B}(L^+G\rtimes\Aut(D))},\uk_{\mathbb{B}(L^+G\rtimes\Aut(D))})$ is a \emph{cohomological correspondence}.\footnote{We refer to \cite[\S4.2]{liu2025higherperiodintegralsderivatives} for a general treatment of cohomological correspondences with kernels. Here, one can forget about the interpretation as a cohomological correspondence and directly work with the Hom space.} Here, \begin{equation}\Cor_{\Hk_{G}^{\loc},\cK\langle-d_{\frc}\rangle}(\uk_{\mathbb{B}(L^+G\rtimes\Aut(D))},\uk_{\mathbb{B}(L^+G\rtimes\Aut(D))})=\Hom^0(\lh_{\loc,!}\cK\langle-d_{\frc}\rangle,\uk_{\mathbb{B}(L^+G\rtimes\Aut(D))})
.\end{equation} 
        \item $ \frd^{\loc}\in \Hom^0(\uk_{\mathbb{B}(L^+G\rtimes\Aut(D))},\rh_{\loc,!}\cK\langle d_{\frd}\rangle)$.\footnote{The element $\mathfrak{d}^{\loc}$ can be regarded as a cohomological correspondence with kernel the Verdier dual of $\cK$.}
    \end{itemize} Here, $d_{\frc},d_{\frd}\in\ZZ$ are some arbitrary integers, which are part of the datum. Define the \emph{degree} of the local volume datum to be the integer $d_v=d_{\frc}+d_{\frd}$.
\end{defn} 

\begin{example}[Local volume datum for determinant line bundle]\label{eg:localvolumedatumdet}
    For each dominant coweight $\lambda\in X_*(T)_+$, we have the local volume datum $v_{\l,\det}=(\IC_{\l}, \frc^{\loc}_{\l,\det},\frd^{\loc}_{\l})$ (see Definition \ref{def:localvolumedatumdet}).
\end{example}

\begin{example}[Cup product with a cohomology class]\label{eg:localvolumedatumcup}
    For each local volume datum $v=(\cK,\frc^{\loc},\frd^{\loc})$ and a cohomology class $\alpha \in H^{|\alpha|}(\Hk_G^{\loc})$, one can define a new local volume datum $\alpha\cdot v:=(\cK,\alpha\cup \frc^{\loc},\frd^{\loc})$ which has degree $d_v+|\alpha|$.
\end{example}

\subsubsection{Global volume data}
We now introduce global volume data.
\begin{defn}\label{def:globalvolumedatum}
    We define a \emph{global volume datum} to be a triple $v=(\cK,\frc,\frd)$ in which: \begin{itemize}
        \item $\cK\in\Shv(\Hk_G)$ is a sheaf supported on $\Hk_{G,\leq\lambda}\sub \Hk_G$ for some $\lambda\in X_*(T)_+$.
        \item $\frc\in\Cor_{\Hk_{G},\cK\langle-d_{\frc}\rangle}(\uk_{\Bun_G\times C},\uk_{\Bun_G\times C})$ is a cohomological correspondence. Here, \begin{equation}\Cor_{\Hk_{G},\cK\langle-d_{\frc}\rangle}(\uk_{\Bun_G\times C},\uk_{\Bun_G\times C})=\Hom^0(\lh_{!}\cK\langle-d_{\frc}\rangle,\uk_{\Bun_G\times C})
.\end{equation} 
        \item $ \frd\in \Hom^0(\uk_{\Bun_G\times C},\rh_{!}\cK\langle d_{\frd}\rangle)$.
    \end{itemize} Here, $d_{\frc},d_{\frd}\in\ZZ$ are some arbitrary integers, which are part of the datum. Define the degree of the global volume datum to be the integer $d_v=d_{\frc}+d_{\frd}$. Here, the maps $\lh,\rh$ are the top maps in \eqref{diag:localtoglo}.
\end{defn} 

\subsubsection{Local-to-global construction}\label{sec:local-to-globalvolume}
From a local volume datum $(\cK,\frc^{\loc},\frd^{\loc})$, we can construct a global volume datum $f^*v=(f_{\Hk}^*\cK,\frc,\frd)$ as follows: Consider the diagram \begin{equation}\label{diag:localtoglo}
    \begin{tikzcd}
        \Bun_G\times C \ar[d, "f"] & \Hk_{G} \ar[r, "\rh"] \ar[l, "\lh"'] \ar[d, "f_{\Hk}"] & \Bun_G\times C \ar[d, "f"] \\
        {\mathbb{B}(L^+G\rtimes\Aut(D))} & \Hk_{G}^{\loc}  \ar[r, "\rh_{\loc}"] \ar[l, "\lh_{\loc}"'] & {\mathbb{B}(L^+G\rtimes\Aut(D))}
    \end{tikzcd}
\end{equation} in which the vertical maps are given by restriction to the formal disc near the point on the curve $C$. The top horizontal row is the global counterpart of the correspondence \eqref{diag:localhecke}. Both squares in \eqref{diag:localtoglo} are Cartesian.

Define \begin{equation}
    \frc:=f^*\frc^{\loc}\in \Cor_{\Hk_{G},f_{\Hk}^*\cK\langle-d_{\frc}\rangle}(\uk_{\Bun_G\times C},\uk_{\Bun_G\times C})=\Hom^0(\lh_!f_{\Hk}^*\cK\langle-d_{\frc}\rangle,\uk_{\Bun_G\times C})
\end{equation} as the composition \[\lh_!f_{\Hk}^*\cK\langle-d_{\frc}\rangle\cong f^*\lh_{\loc,!}\cK\langle-d_{\frc}\rangle\xrightarrow{f^*\mathfrak{c}^{\loc}}f^*\uk_{\mathbb{B}(L^+G\rtimes\Aut(D))}\cong \uk_{\Bun_G\times C},\]

and \begin{equation}
    \frd:=f^*\frd^{\loc}\in \Hom^0(\uk_{\Bun_G\times C},\rh_!f_{\Hk}^*\cK\langle d_{\frd}\rangle)
\end{equation} as the composition \[\uk_{\Bun_G\times C}\cong f^*\uk_{\mathbb{B}(L^+G\rtimes\Aut(D))}\xrightarrow{f^*\mathfrak{d}^{\loc}} f^*\rh_{\loc,!}\cK\langle d_{\frd}\rangle\cong \rh_!f_{\Hk}^*\cK\langle d_{\frd}\rangle.\] Throughout the article, we will always work with global volume data coming from local volume data as above.

\subsection{Relative Hecke operators}
In this section, we will introduce relative Hecke operators associated with volume data in both local and global settings.

\begin{itemize}
    \item In \S\ref{sec:globalrelhecke}, we introduce global relative Hecke operators, which will exactly generalize the operators in \eqref{eq:gammaintrodet}.
    \item In \S\ref{sec:localrelhecke}, we introduce local relative Hecke operators, which serve as local models for the global ones.
\end{itemize}
\subsubsection{Global relative Hecke operators}\label{sec:globalrelhecke}
Now we introduce global relative Hecke operators associated to a global volume datum $v=(\cK,\frc,\frd)$.

We have a map \begin{equation}\label{eq:globhmapc}
\Gamma_{c,v}^C: H^*_c(\Bun_G\times C)\xrightarrow{H_c^*(\frd)} H^{*+d_{\frd}}_c(\Hk_{G},\cK)\xrightarrow{H_c^*(\frc)} H^{*+d_v}_c(\Bun_G\times C)
.\end{equation} 

The (graded) dual of $H_c^*(\Bun_G\times C)$ is $H^*(\Bun_G\times C)(\dim\Bun_G+1)$. Ignoring the Tate twist, the (graded) dual of \eqref{eq:globhmapc} is \begin{equation}\label{eq:globhmap}
  \Gamma_{v}^C:  H^*(\Bun_G\times C)\xrightarrow{H^*(\frc)} H_{-(*+d_{\frc})}^{BM}(\Hk_G/\Bun_G\times C,\cK) \xrightarrow{H^{*}(\frd)} H^{*+d_v}(\Bun_G\times C)
\end{equation} in which $H_{-*}^{BM}(\Hk_G/\Bun_G\times C,\cK)=\Hom^*(\cK,\om_{\Hk_G/\Bun_G\times C})$. Here, we are using \[\om_{\Hk_G/\Bun_G\times C}=\lh^!\uk_{\Bun_G\times C}\cong \rh^!\uk_{\Bun_G\times C}.\]
The maps in \eqref{eq:globhmap} are defined by $H^*(\frc)=H_c^*(\frc)^*$, $H^*(\frd)=H_c^*(\frd)^*$. More explicitly, the first map can be identified with the map
\[\begin{split}H^*(\frc): H^*(\Bun_G\times C)&=\Hom^*(\uk_{\Bun_G\times C},\uk_{\Bun_G\times C})\\ &\to \Hom^*(\lh_!\cK\langle -d_{\frc}\rangle,\uk_{\Bun_G\times C}) \\ &\cong \Hom^*(\cK\langle -d_{\frc}\rangle,\lh^!\uk_{\Bun_G\times C})\\ &=H_{-(*+d_{\frc})}^{BM}(\Hk_G/\Bun_G\times C,\cK)\end{split}.\] The second map can be identified with \[\begin{split}
    H^*(\frd):H_{-*}^{BM}(\Hk_G/\Bun_G\times C,\cK)&\cong \Hom^*(\cK,\rh^!\uk_{\Bun_G\times C})\\& \cong \Hom^*(\rh_!\cK,\uk_{\Bun_G\times C}) \\  &\to \Hom^*(\uk_{\Bun_G\times C}\langle -d_{\frd}\rangle,\uk_{\Bun_G\times C})\\ &=H^{*+d_{\frd}}(\Bun_G\times C)
\end{split}.\]  The maps $H^*(\frc), H^*(\frd)$ are linear over $H^*(\Bun_G\times C)$. That is, for any $\alpha\in H^*(\Bun_G\times C)$, we have \begin{equation}
    H^*(\frc)(\alpha\cdot-)=(\lh^*\alpha)\cdot H^*(\frc)(-)
\end{equation} and \begin{equation}
    H^*(\frd)((\rh^*\alpha)\cdot -)=\alpha\cdot H^*(\frd)(-)
.\end{equation}

We further define the map \begin{equation}\label{eq:gammacdef}\Gamma_{c,v}:H^*_c(\Bun_G)\xrightarrow{\pr_1^*} H^*_c(\Bun_G\times C)\xrightarrow{\Gamma_{c,v}^C}H^{*+d_v}_c(\Bun_G\times C)\xrightarrow{\pr_{1,!}} H^{*+d_v-2}_c(\Bun_G)\end{equation} and its dual \begin{equation}\label{eq:gammadef}
    \Gamma_{v}:H^*(\Bun_G)\xrightarrow{\pr_1^*} H^*(\Bun_G\times C)\xrightarrow{\Gamma_{v}^C}H^{*+d_v}(\Bun_G\times C)\xrightarrow{\pr_{1,!}} H^{*+d_v-2}(\Bun_G).\end{equation}

    We call operators defined in \eqref{eq:globhmapc}\eqref{eq:globhmap}\eqref{eq:gammacdef}\eqref{eq:gammadef} \emph{global relative Hecke operators}.

\begin{example}\label{eg:globalrelheckeforvolume}
    When $v=f^*v_{\lambda,\det}$ where $v_{\lambda,\det}$ is in Example \ref{eg:localvolumedatumdet}, we have $\Gamma_{c,f^*v_{\lambda,\det}}=\Gamma_{c,\cL_{\det,I}}$ where the latter is defined in \eqref{eq:gammaintrodet} in which one takes $r=1$ and $\lambda_I=(\lambda)$.
\end{example}

\subsubsection{Local relative Hecke operators}\label{sec:localrelhecke}
Now we introduce local relative Hecke operators associated to a local volume datum $v=(\cK,\frc^{\loc},\frd^{\loc})$ as defined in Definition \ref{def:localvolumedatum}. They will play an important role in the computation of global relative Hecke operators and in formulating the main theorem.

Note that \eqref{eq:globhmap} has a local counterpart 
\begin{equation}\label{eq:localhmap}
     \Gamma_{v}^{\loc}:   H^*(\mathbb{B}(L^+G\rtimes\Aut(D)))\xrightarrow{H^*(\frc^{\loc})} H_{-(*+d_{\frc})}^{BM}(\Hk_G^{\loc}/\mathbb{B}(L^+G\rtimes\Aut(D)),\cK) \xrightarrow{H^*(\frd^{\loc})} H^{*+d_v}(\mathbb{B}(L^+G\rtimes\Aut(D)))
.\end{equation} Here, the maps $H^*(\frc^{\loc})$ and $H^*(\frd^{\loc})$ are linear over $H^*(\mathbb{B}(L^+G\rtimes\Aut(D)))$. We call operators $\Gamma_v^{\loc}$ defined in \eqref{eq:localhmap} \emph{local relative Hecke operators}.

\subsubsection{Comparison with \cite{FYZvolume}}\label{sec: comparision with FYZ}
Assume that \(\lambda\) is a dominant minuscule coweight. Then we have
\[
  \Hk_{G,\leq\lambda}^{\loc}
  \simeq
  \left[
    \bigl(L^+Gt^\lambda L^+G\bigr)
    \big/
    \bigl((L^+G\times L^+G)\rtimes\Aut(D)\bigr)
  \right]
  \simeq \mathbb B K_\lambda,
\]
where the action is given by \((g_1,g_2,\varphi)\cdot h=g_1\varphi(h)g_2^{-1}\), and
\[
  K_\lambda:=
  \left\{
    (g_1,g_2,\varphi)\in
    (L^+G\times L^+G)\rtimes\Aut(D)
    \ \middle|\
    g_1\varphi(t)^\lambda g_2^{-1}=t^\lambda
  \right\}.
\]
The restrictions of \(\lh_{\loc}\) and \(\rh_{\loc}\) to
\(\Hk_{G,\leq\lambda}^{\loc}\) are induced by the maps
\(K_\lambda\rightarrow L^+G\rtimes\Aut(D)\) that forget the second and first \(L^+G\)-factors, respectively. Set \(L^+G_\lambda:=L^+G\cap t^\lambda L^+Gt^{-\lambda}\) and \(L^+G_{-\lambda}:=L^+G\cap t^{-\lambda}L^+Gt^\lambda\). The two forgetful maps induce isomorphisms
\[
  \overleftarrow{t}:K_\lambda\xrightarrow{\sim}L^+G_\lambda\rtimes\Aut(D),
  \qquad
  \overrightarrow{t}:K_\lambda\xrightarrow{\sim}L^+G_{-\lambda}\rtimes\Aut(D).
\]
Let \(\ev:L^+G\rightarrow G\), \(g\mapsto g(0)\), be the evaluation map. Then \(\ev(L^+G_\lambda)=P_{-\lambda}\) and \(\ev(L^+G_{-\lambda})=P_\lambda\). We have the following Cartesian diagrams:

\[
\begin{tikzcd}
    {\mathbb{B}(L^+G\rtimes\Aut(D))}
    & {\Hk_{G,\leq\lambda}^{\loc}=\mathbb{B}(K_\lambda)} \\
    {\mathbb{B}(G\times\mathbb{G}_m)}
    & {\mathbb{B}(P_{-\lambda}\times\mathbb{G}_m)}
    \arrow["\lh_{\loc}"', from=1-2, to=1-1]
    \arrow[from=1-1, to=2-1]
    \arrow[from=1-2, to=2-2]
    \arrow[from=2-2, to=2-1]
\end{tikzcd}
\]

\[
\begin{tikzcd}
    {\Hk_{G,\leq\lambda}^{\loc}=\mathbb{B}(K_\lambda)}
    & {\mathbb{B}(L^+G\rtimes\Aut(D))} \\
    {\mathbb{B}(P_{\lambda}\times\mathbb{G}_m)}
    & {\mathbb{B}(G\times\mathbb{G}_m)}
    \arrow["\rh_{\loc}", from=1-1, to=1-2]
    \arrow[from=1-1, to=2-1]
    \arrow[from=1-2, to=2-2]
    \arrow[from=2-1, to=2-2]
\end{tikzcd}
\]

A local volume datum of the form \(v=(\IC_\lambda,\frc^{\loc},\frd^{\loc})\) is specified by
\begin{align*}
\frc^{\loc}&\in \Cor_{\Hk_G^{\loc},\IC_\lambda\langle-d_{\frc}\rangle}(\uk_{\mathbb{B}(L^+G\rtimes\Aut(D))},\uk_{\mathbb{B}(L^+G\rtimes\Aut(D))})=H^{d_\lambda+d_{\frc}}(\mathbb{B}(P_\lambda\times\mathbb{G}_m),k),\\
\frd^{\loc}&\in \Hom^0(\uk_{\mathbb{B}(L^+G\rtimes\Aut(D))},\rh_{\loc,!}\IC_\lambda\langle d_{\frd}\rangle)=H^{d_\lambda+d_{\frd}}(\mathbb{B}(P_\lambda\times\mathbb{G}_m),k).
\end{align*}
In particular, take \(d_{\frc}=d_\lambda+2\) and \(d_{\frd}=-d_\lambda\). Let \(\eta\in H^{2d_\lambda+2}(\mathbb{B}(P_\lambda),k)\) and \(\eta'\in H^{2d_\lambda}(\mathbb{B}(P_\lambda),k)\). Then \(\eta+\frac{1}{2-2g}\eta'\hbar\) defines a local cohomological correspondence \(\frc^{\loc}\) when $g\neq 1$. Take \(\frd^{\loc}\) to be the fundamental class (the identity). This exactly recovers the basic setup of \cite[\S 5.1.2]{FYZvolume} when $g\neq 1$.

\subsection{Calculating local relative Hecke operators}
Now we study local relative Hecke operators in more detail.

\subsubsection{Cohomology of $\mathbb{B}(L^+G\rtimes\Aut(D))$}\label{sec:cohbg}
Let $R=H^*(\mathbb{B}T)$. Then $H^*(\mathbb{B}G)=R^W$ and $H^*(\mathbb{B}(L^+G\rtimes\Aut(D)))=R^W[\hbar]$, where $\hbar:=c_1(T_D)\in H^2([D/\Aut(D)])\cong H^2(\mathbb{B}\Aut(D))$. Here $T_D$ is the tangent bundle of $D$.

Note that $H^*(\mathbb{B}G)$ has a natural augmentation ideal $H^{>0}(\mathbb{B}G)\sub H^*(\mathbb{B}G)$. This defines a decreasing \emph{augmentation filtration} $\{F_{\aug}^{\bl}H^*(\mathbb{B}G) \}_{\bl\in \ZZ_{\geq 0}}$ on $H^*(\mathbb{B}G)$ such that \[F_{\aug}^{i}H^*(\mathbb{B}G):=(H^{>0}(\mathbb{B}G))^i\sub  H^{*}(\mathbb{B}G).\] Define the \emph{Gross motive} \begin{equation}\label{eq:grossmotive}\mathbb{V}:=\Gr_{\aug}^1H^*(\mathbb{B}G).\end{equation} We have a canonical isomorphism $\Gr^{\bl}_{\aug}(H^*(\mathbb{B}G))\cong \Sym^{\bl}(\mathbb{V})$.


\subsubsection{Local relative Hecke operator on associated graded}
We keep working with a local volume datum $v=(\cK,\frc^{\loc},\frd^{\loc})$ and assume $d_v=2$.

Consider maps \[\lh^*_{\loc},\rh_{\loc}^*:H^*(\mathbb{B}(L^+G\rtimes\Aut(D)))\to H^*(\Hk_{G}^{\loc})\] where $\lh_{\loc},\rh_{\loc}$ are the bottom maps in \eqref{diag:localtoglo}. For $\alpha\in H^{|\alpha|}(\mathbb{B}(L^+G\rtimes\Aut(D)))$, define \begin{equation}\label{eq:localcommutator} [\alpha]=\hbar^{-1}\cdot (\lh_{\loc}^*\alpha-\rh^*_{\loc}\alpha)\in H^{|\alpha|-2}(\Hk_G^{\loc}). \end{equation} Here, we are using the facts that $H^*(\Hk_G^{\loc})$ is a flat $k[\hbar]$-module and $\lh_{\loc}^*\alpha-\rh^*_{\loc}\alpha =0 \mod \hbar$, both of which follow from \cite[Theorem\,1]{BF}.

Consider \begin{equation}\label{eq:nablatildeloc}\widetilde{\nabla}_v^{\loc}:= \Gamma_{[-]\cdot v}^{\loc}(1) \in \End(H^*(\mathbb{B}(L^+G\rtimes\Aut(D))))\end{equation} where the dot product in $[-]\cdot v$ is understood as in Example \ref{eg:localvolumedatumcup}. Since $d_v=2$, we know $\Gamma_v^{\loc}(1)\in H^2(\mathbb{B}(L^+G\rtimes\Aut(D)))=k\cdot \hbar$. Write \begin{equation}\label{eq:bdef}
    \Gamma_v^{\loc}(1) = -b_v\hbar
\end{equation} for $b_v\in k$.

For any $\alpha \in H^*(\mathbb{B}(L^+G\rtimes\Aut(D)))$, note that \begin{equation}\label{eq:gammaformulaloc1}
\begin{split}
    \Gamma_v^{\loc}(\alpha) & = H^*(\frd^{\loc})\circ H^*(\frc^{\loc})(\alpha) \\
    & = H^*(\frd^{\loc})((\lh_{\loc})^*(\alpha)H^*(\frc^{\loc})1) \\
    & = H^*(\frd^{\loc})(((\rh_{\loc})^*(\alpha)+\hbar[\alpha])H^*(\frc^{\loc})1) \\
    & = \alpha\Gamma_v^{\loc}(1)+\hbar\cdot\Gamma_{[\alpha]\cdot v}^{\loc}(1) \\
    & = -b_v\hbar\cdot \alpha+\hbar\cdot\widetilde{\nabla}_v^{\loc}(\alpha)
    \end{split}
.\end{equation}

Define $\nabla_v^{\loc}:=\widetilde{\nabla}_v^{\loc}\otimes_{k[\hbar]}k\in \End(H^*(\mathbb{B}G))$. We have the following important observation:

\begin{lemma}\label{lem:localder}
    The map $\nabla_v^{\loc}:H^*(\mathbb{B}G)\to H^*(\mathbb{B}G)$ is a derivation. Moreover, it preserves the natural augmentation filtration $F_{\mathrm{aug}}^{\bl}H^*(\mathbb{B}G)$ defined in \S\ref{sec:cohbg}.
\end{lemma}

\begin{proof}
    Both follow from the computation that \begin{equation}\label{eq:gammaformulaloc} \widetilde{\nabla}_v^{\loc}(\alpha\beta)=\alpha\widetilde{\nabla}_v^{\loc}(\beta)+\widetilde{\nabla}_v^{\loc}(\alpha)\beta+\hbar\Gamma_{[\alpha][\beta]\cdot v}^{\loc}(1) \end{equation} for any $\alpha,\beta\in H^*(\mathbb{B}(L^+G\rtimes\Aut(D)))$.
\end{proof}

Given Lemma \ref{lem:localder}, it is natural to consider
\begin{equation}\label{eq:tmap} E_v:=\Gr^1_{\mathrm{aug}}\nabla_v^{\loc}:\mathbb{V}\to\mathbb{V}.\end{equation} We would like to call this the \emph{relative Hecke operator on Gross motive}. Negatives of its eigenvalues form a (multi-)set \begin{equation}\label{eq:eigenweightsset}\{\epsilon_{v,i} \mid 1\leq i\leq \dim\mathbb{V}\} \end{equation} and are called \emph{eigenweights} of $v$. See \S\ref{sec:eigenweights} for an ordering of the eigenweights when $G$ is almost simple.

We have the following proposition to be compared with its global counterpart, Proposition \ref{prop:grgammaformula}:
\begin{prop}\label{prop:grgammaformulaloc}
    The induced map $\overline{\Gamma}_v^{\loc}:=(\hbar^{-1}\Gamma_v^{\loc})\otimes_{k[\hbar]}k$ preserves the filtration  $F_{\mathrm{aug}}^{\bl}H^*(\mathbb{B}G)$. Under the canonical isomorphism $\Gr^{\bl}_{\aug}H^*(\mathbb{B}G)\cong \Sym^{\bl}(\mathbb{V})$, we have $\Gr^{\bl}_{\aug}\overline{\Gamma}_v^{\loc}=-b_v\cdot \id + \nabla_{E_v}\in \End(\Sym^{\bl}(\mathbb{V}))$. Here $\nabla_{E_v}$ is the derivation on $\Sym^{\bl}(\mathbb{V})$ such that $\nabla_{E_v}|_{\mathbb{V}}=E_v$.
\end{prop}

\begin{example}\label{eg:Edet}
When $v=v_{\lambda,\det}$ as in Example \ref{eg:localvolumedatumdet}, we have $\epsilon_{\lambda,i}=\epsilon_{v_{\lambda,\det},i}$ and $b_{\lambda} = b_{v_{\lambda,\det}}$. This gives the numbers involved in Theorem \ref{thm:mainsame}.
\end{example}

\subsection{Volume of Shtukas: a general formulation}
Now we formulate a general version of Theorem \ref{thm:maindiff}.

\subsubsection{Definition of volume}
For each local volume datum $v=(\cK,\frc^{\loc},\frd^{\loc})$ as in Definition \ref{def:localvolumedatum} and $e\in\pi_1(G)\cong \pi_0(\Bun_G)$, assuming $\cK\in\Shv(\Hk_{G}^{\loc,e_{\cK}})$ for some $e_{\cK}\in\pi_1(G)$,\footnote{One has connected components decomposition $\Hk_G^{\loc}=\coprod_{e\in \pi_1(G)}\Hk_G^{\loc,e}$} we consider the corresponding global volume datum $f^*v$ as defined in \S\ref{sec:local-to-globalvolume}.

We have the global relative Hecke operators \begin{equation}
    \Gamma_{c,f^*v}^e:H^*_c(\Bun_G^e)\to H^{*+d_v-2}_c(\Bun_G^{e+e_{\cK}})
\end{equation} and \begin{equation}
    \Gamma_{f^*v}^e: H^{*}(\Bun_G^{e+e_{\cK}})\to H^{*+d_v-2}(\Bun_G^{e})
\end{equation} which are given by restricting those in \S\ref{sec:globalrelhecke} to one connected component.

Fix $I=\{1,2,\cdots,r\}$ for some $r\in\ZZ_{\geq 0}$ and choose a sequence of local volume data $v_I:=(v_i=(\cK_i\in\Shv(\Hk_G^{\loc,e_i}),\frc^{\loc}_i,\frd^{\loc}_i))_{i\in I}$ such that $d_{v_i}=d_{\frc_i}+d_{\frd_i}=2$ for any $i\in I$. Let $e_{v_I}=\sum_i e_{i}\in\pi_1(G)$. Define the global relative Hecke operator attached to $v_I$ \begin{equation}\Gamma_{c,f^*v_I}^e:=\Gamma_{c,f^*v_1}^{e+e_{v_I}-e_1}\circ\cdots\circ\Gamma_{c,f^*v_r}^e:H^*_c(\Bun_G^e)\to H_c^*(\Bun_G^{e+e_{v_I}})\end{equation} and its dual \begin{equation}
    \Gamma_{f^*v_I}^e:=\Gamma_{f^*v_r}^e\circ\cdots\circ\Gamma_{f^*v_1}^{e+e_{v_I}-e_1} : H^*(\Bun_G^{e+e_{v_I}})\to H^*(\Bun_G^e)
.\end{equation} 

Moreover, we make the following assumption:
\begin{assumption}\label{assumption:rational}
    Each local volume datum $v_i$ is defined over a subfield $E\sub k$ such that $E/\QQ$ is a finite extension. That is, the sheaf $\cK_i$ is defined over $E$ and the maps $\frc_i^{\loc},\frd_i^{\loc}$ are defined over $E$.
\end{assumption}

Define \begin{equation}
    \mathrm{vol}(\Sht_{G,I}^e,f^*v_I):=\begin{cases}
    0 &, e_{v_I}\neq 0 \\
        \tr(\Frob\circ \Gamma_{c,f^*v_I}^e,H_c^*(\Bun_G^e)) &, e_{v_I}=0
    \end{cases}
.\end{equation} 

One can also define the volume for all connected components \begin{equation}\mathrm{vol}(\Sht_{G,I},f^*v_I):=\tr(\Frob\circ \Gamma_{c,v_I},H_c^*(\Bun_G)).\end{equation} Then one has $\mathrm{vol}(\Sht_{G,I},f^*v_I)=\sum_{e\in \pi_1(G)}\mathrm{vol}(\Sht_{G,I}^e,f^*v_I)$.

\begin{remark}\label{rmk:converg}
The trace above should be understood as the summation
$
\tr(\Frob\circ \Gamma_{c,f^*v_I},H^*_c(\Bun_G))=\sum_{i\in\ZZ}(-1)^i\tr(\Frob\circ \Gamma_{c,f^*v_I},H^i_c(\Bun_G)),
$
in which each vector space $H^i_c(\Bun_G)$ is finite-dimensional. Under Assumption \ref{assumption:rational}, this sum is absolutely convergent under any identification $\overline{\mathbb{Q}}_\ell\cong \mathbb{C}$. This can be proved in the same way as \cite[Proposition\,5.5.3]{FYZvolume}.
\end{remark}

\begin{example}\label{eg:vol=volgen}
    When taking $v_I=(v_{\lambda_1,\det},\cdots,v_{\lambda_r,\det})$ where each term is as in Example \ref{eg:localvolumedatumdet} for $\lambda_I\in X_*(T)^r$, we have $\mathrm{vol}(\Sht_{G,I},f^*v_I)=\mathrm{vol}(\Sht_{G,\leq\lambda_I},f_{\Sht,I}^*\cL_{\det,I})$ where the latter is considered in Theorem \ref{thm:maindiff}.
\end{example}

\subsubsection{Ordering the eigenweights}\label{sec:eigenweights}
Recall the (multi-)set of eigenweights
$
\{\epsilon_{v,i}\mid 1\leq i\leq n\}
$ defined in \eqref{eq:eigenweightsset}. When $G$ is almost simple, we refine this multiset by partially introducing an ordering using the decomposition $\mathbb{V}=\bigoplus_{i=1}^n\overline{\mathbb{Q}}_\ell(-d_i)$, where $\{d_i\mid 1\leq i\leq n\}$ are as in Theorem \ref{thm:mainsame}.

 When $G$ is not of type $D_{n}$ for $n$ even, the numbers $\{d_i, 1\leq i\leq n\}$ are distinct and we order them such that $d_1<d_2<\cdots<d_n$. Define $\mathbb{V}_i:=\Qlbar(-d_i)$, and we have $\mathbb{V}=\bigoplus_{i=1}^n\mathbb{V}_i$. We define $E_{v}|_{\mathbb{V}_i}=-\epsilon_{v,i}\cdot\id_{\mathbb{V}_i}$. In this case, the matrix $E_{v}$ is diagonal under the decomposition $\mathbb{V}=\bigoplus_{i=1}^n\mathbb{V}_i$.

When $G$ is of type $D_n$ for $n$ even, we take $d_i=2i$ for $1\leq i\leq n-1$ and $d_n=n$. When $i\neq n/2,n$, we define $\mathbb{V}_i:=\Qlbar(-d_i)$. When $i\in \{n/2,n\}$, the degree $n$ part of $\mathbb{V}$ is a two dimensional vector space. We fix a subgroup $S_2\sub \Out(\frgc)$ which acts on $\mathbb{V}$ through pinned automorphisms.\footnote{When $n=4$, such a choice is equivalent to choosing an outer node of the Dynkin diagram, or equivalently, to choosing a standard representation.} We define $\mathbb{V}_{n/2}$ as the $S_2$-invariant part of the degree $n$ part of $\mathbb{V}$, and define $\mathbb{V}_n$ as the part that $S_2$ acts via the non-trivial character.\footnote{The one-dimensional vector space $\mathbb{V}_n$ is spanned by the \emph{Pfaffian}.} We still have $\mathbb{V}=\bigoplus_{i=1}^n\mathbb{V}_i$. For $i\neq n/2,n$, we define $E_{v}|_{\mathbb{V}_i}=-\epsilon_{v,i}\cdot\id_{\mathbb{V}_i}$. When $i\in \{n/2,n\}$, we only define the (multi-)set $\{\epsilon_{v,n/2},\epsilon_{v,n}\}$ as the set of eigenvalues of the $2\times2$-matrix $-E_{v}|_{\mathbb{V}_{n/2}\oplus\mathbb{V}_{n}}$. In this case, the matrix $E_{v}$ is block diagonal under the decomposition $\mathbb{V}=\bigoplus_{i=1}^n\mathbb{V}_i$ with a unique $2\times 2$-block $E_{v}|_{\mathbb{V}_{n/2}\oplus \mathbb{V}_n}$.

\subsubsection{Main result}
Now we formulate a general version of Theorem \ref{thm:maindiff}. Take $v_I$ as before. We make the following assumption.
\begin{assumption}\label{assumption:paircomm}
    We assume that $\{E_{v_i}\in\End(\mathbb{V})\}_{i\in I}$ pairwise commute. Here $E_{v_i}$ are defined in \eqref{eq:tmap}.
\end{assumption}

Note that Assumption \ref{assumption:paircomm} is an empty assumption when $G$ is almost simple and not of type $D_n$ for $n$ even. When $G$ is of type $D_n$ for $n$ even, Assumption \ref{assumption:paircomm} amounts to requiring the $2\times 2$-matrices $E_{v_i}|_{\mathbb{V}_{n/2}\oplus \mathbb{V}_n}$ to be mutually commutative. Under this assumption, we can simultaneously upper-triangularize the matrices $E_{v_i}|_{\mathbb{V}_{n/2}\oplus \mathbb{V}_n}$. We can order each set $\{\epsilon_{v_i,n/2},\epsilon_{v_i,n}\}$ such that each number $\epsilon_{v_i,n}$ is an eigenvalue for a joint eigenvector of all $-E_{v_i}|_{\mathbb{V}_{n/2}\oplus \mathbb{V}_n}$.

Define the differential operator on $\mathbb{R}^n$ \begin{equation}\label{eq:diff} D_{v_i}=(2g-2)b_{v_i}+(\log q)^{-1}\sum_{j=1}^n\epsilon_{v_i,j}\partial_{s_j}\end{equation} where $b_{v_i}$ are defined in \eqref{eq:bdef}.


\begin{thm}\label{thm:main}
    Let $v_I=\{v_i\}_{i=1,\cdots,r}$ be a collection of local volume data such that $d_{v_i}=2$ for each $i\in I$ and $e_{v_I}=0$. Under Assumption \ref{assumption:paircomm} and Assumption \ref{assumption:rational}, for each $e\in\pi_1(G)$, we have \begin{equation}
        \mathrm{vol}(\Sht_{G,I}^e,f^*v_I)=q^{\dim\Bun_G} \big((\prod_{i=1}^r D_{v_i})\mathscr{L}_{C,G}(s_1,\cdots,s_n)\big)\big|_{s_1=\cdots=s_n=0}
    .\end{equation}
\end{thm}

In the case of Example \ref{eg:vol=volgen}, Theorem \ref{thm:main} becomes Theorem \ref{thm:maindiff}.

The proof of Theorem \ref{thm:main} will be given in \S\ref{sec:proofmain}.

\begin{remark}
Assumption \ref{assumption:paircomm} can fail. See \cite[Remark\,1.3.9]{feng2026eigenweightsarithmetichirzebruchproportionality} for a counterexample. This also follows from our computation of $E_{v_{\lambda,\det}}$ in the type $D$ case in \S\ref{sec:Depsilon}.
\end{remark}

\section{Calculating global relative Hecke operators}\label{sec:proofmain}
This section is devoted to the proof of Theorem \ref{thm:main}, which will be given in \S\ref{sec:grgamma}.

\subsection{\texorpdfstring{Cohomology of $\Bun_G$}{Cohomology of BunG}} We now recall the Atiyah--Bott formula for the cohomology ring $H^*(\Bun_G)$ and the filtrations on it as considered in \cite{FYZvolume}.

\subsubsection{Augmentation filtration}\label{sec:augfil}
Recall the connected component decomposition $\Bun_G=\coprod_{e\in \pi_1(G)}\Bun_G^e$ where $\pi_1(G)$ is the algebraic fundamental group of $G$. For each $e\in \pi_1(G)$, the cohomology ring $H^*(\Bun_G^e)$ has a natural augmentation ideal $H^{>0}(\Bun_G^e)\sub H^*(\Bun_G^e)$. We get a decreasing augmentation filtration $\{F_{\aug}^{\bl}H^*(\Bun_G^e) \}_{\bl\in \ZZ_{\geq 0}}$ on $H^*(\Bun_G^e)$ such that \[F_{\aug}^{i}H^*(\Bun_G^e):=(H^{>0}(\Bun_G^e))^i\sub  H^{*}(\Bun_G^e).\] 

\subsubsection{Atiyah--Bott formula}\label{sec:AB}
Consider the map \[f:\Bun_G\times C\to \mathbb{B}(L^+G\rtimes\Aut(D)).\] It induces a map\footnote{We write $H_{\bullet}(C)$ instead of $H_*(C)$ to match the homological grading here with the Ran grading in \S\ref{sec:Ranfil}.} \[\taut: H^*(\mathbb{B}(L^+G\rtimes\Aut(D)))\otimes H_{\bullet}(C)\to H^*(\Bun_G)\] defined as \[\taut(\gamma\otimes z)= \int_zf^*\gamma\] where $\int_z: H^{i}(C)\to k$ is the natural integration map and $\pr_1^{\Bun_G}:\Bun_G\times C\to \Bun_G$ is the projection.

For each $e\in \pi_1(G)$, one can further restrict to a single connected component and get \[\taut^e: H^*(\mathbb{B}(L^+G\rtimes\Aut(D)))\otimes H_{\bullet}(C)\to H^*(\Bun_G^e).\] We abbreviate $\taut(\gamma\otimes z)=\gamma^z$. We call classes of the form $\gamma^z$ \emph{tautological classes}.

Note that the map \[\taut^e|_{H^*(\mathbb{B}G)\otimes H_{\bullet}(C)}:H^*(\mathbb{B}G)\otimes H_{\bullet}(C)\to H^*(\Bun_G^e)\] preserves augmentation filtrations on both sides. Here, we are considering the filtration $F_{\aug}^{i}(H^*(\mathbb{B}G)\otimes H_{\bullet}(C)):= F_{\aug}^{i}H^*(\mathbb{B}G)\otimes H_{\bullet}(C)$ on the left-hand-side. Therefore, we get an induced homomorphism \[\Gr^{\bl}_{\aug}(\taut^e|_{H^*(\mathbb{B}G)\otimes H_{\bullet}(C)}):\Gr_{\aug}^{\bl}(H^*(\mathbb{B}G)\otimes H_{\bullet}(C))\to \Gr_{\aug}^{\bl} H^*(\Bun_G^e).\] The Atiyah--Bott formula can be formulated as follows:
\begin{thm}\label{thm:AB}
    The map $\Gr^{\bl}_{\aug}(\taut^e|_{H^*(\mathbb{B}G)\otimes H_{\bullet}(C)})|_{\mathbb{V}\otimes H_{\bullet}(C)}:\mathbb{V}\otimes H_{\bullet}(C)\to \Gr_{\aug}^{\bl} H^*(\Bun_G^e)$ induces a ring isomorphism \[\mathrm{AB}^e:\Sym^{\bl}(\mathbb{V}\otimes H_{\bullet}(C))\isom \Gr_{\aug}^{\bl}H^*(\Bun_G^e) .\]
\end{thm}

Theorem \ref{thm:AB} was first proved in \cite{heinloth2010cohomology} in positive characteristic. See \cite[Theorem\,4.2.8]{FYZvolume} for this formulation and a generalization.


\subsubsection{Ran filtration}\label{sec:Ranfil}
Now we introduce the Ran grading considered in \cite[\S5.5.4]{FYZvolume}. For each $e\in \pi_1(G)$, define the \emph{Ran filtration} $\{F^{\Ran}_{\bullet}H^*(\Bun_G^e)\}_{\bullet\in \ZZ_{\geq 0}}$ such that $F^{\Ran}_iH^*(\Bun_G^e)\sub H^*(\Bun_G^e)$ is spanned by those elements of the form $\gamma_1^{z_1}\cdots\gamma_k^{z_k}$ such that $\sum_s|z_s|\leq i$. Here, we have $\gamma_s\in H^*(\mathbb{B}(L^+G\rtimes\Aut(D))), z_s\in H_{\bullet}(C)$, and $|z_s|$ is the homological degree of $z_s$, and $k$ is any non-negative integer.

Note that the Ran filtration is multiplicative, therefore, the associated graded $\Gr^{\Ran}_{\bullet}H^*(\Bun_G^e)$ carries a natural commutative ring structure and a natural augmentation filtration $\{F_{\aug}^{\bl}\Gr^{\Ran}_{\bullet}H^*(\Bun_G^e)\}_{\bl\in \ZZ_{\geq 0}}$.

\subsection{Commutator relation}
We first introduce the key ingredient in the proof of Theorem \ref{thm:main}: Theorem \ref{thm:universalcommutator}, which is a special case of the commutator relations conjectured in \cite[\S4.5]{liu2025higherperiodintegralsderivatives}.

We have a commutative diagram \begin{equation}
    \begin{tikzcd}
        H^*(\mathbb{B}(L^+G\rtimes\Aut(D))) \ar[r, "H^*(\frc^{\loc})"] \ar[d, "f^*"] & H_{-(*+d_{\frc})}^{BM}(\Hk_G^{\loc}/\mathbb{B}(L^+G\rtimes\Aut(D)),\cK) \ar[r, "H^*(\frd^{\loc})"] \ar[d, "f_{\Hk}^*"] & H^{*+d_v}(\mathbb{B}(L^+G\rtimes\Aut(D))) \ar[d, "f^*"] \\
        H^*(\Bun_G\times C) \ar[r, "H^*(\frc)"] & H_{-(*+d_{\frc})}^{BM}(\Hk_G/\Bun_G\times C,f_{\Hk}^*\cK) \ar[r, "H^*(\frd)"] & H^{*+d_v}(\Bun_G\times C)
    \end{tikzcd}
.\end{equation}

The map $f^*:R^W[\hbar]=H^*(\mathbb{B}(L^+G\rtimes\Aut(D)))\to H^*(\Bun_G\times C)$ admits the following explicit description: Consider the map $\ev:\Bun_G\times C\to \mathbb{B}G$ and the induced map $\ev^*:R^W=H^*(\mathbb{B}G)\to H^*(\Bun_G\times C)$. We have \begin{equation}
    f^*\alpha=\ev^*\alpha
\end{equation} for $\alpha\in  R^W \sub R^W[\hbar]$ and \begin{equation} f^*\hbar=1\otimes c_1(T_C)=1\otimes(2-2g)\xi\in H^*(\Bun_G\times C).\end{equation} Here $\xi\in H^2(C)$ is the fundamental class of $C$.

Consider the maps \[\lf:\Hk_G\times C\xrightarrow{\lh\times\id}\Bun_G\times C\times C \xrightarrow{\pr_{1,3}}\Bun_G\times C\xrightarrow{f} \mathbb{B}(L^+G\rtimes\Aut(D))\] and \[\rf:\Hk_G\times C\xrightarrow{\rh\times\id}\Bun_G\times C\times C \xrightarrow{\pr_{1,3}}\Bun_G\times C\xrightarrow{f} \mathbb{B}(L^+G\rtimes\Aut(D)).\] Consider the map $\D_{\Hk}=(\id,l):\Hk_G\to \Hk_G\times C$ where $l:\Hk_G\to C$ is the map remembering the leg. We have a Gysin homomorphism \[\D_{\Hk,!}:H^*(\Hk_G)\to H^{*+2}(\Hk_G\times C)\] and the map $f_{\Hk}^*:H^*(\Hk_G^{\loc})\to H^*(\Hk_G)$.

\begin{thm}\label{thm:universalcommutator}
    For any $\alpha\in H^i(\mathbb{B}(L^+G\rtimes\Aut(D)))$, we have \begin{equation}\label{eq:universalcommutator}\lf^*\alpha-\rf^*\alpha=\D_{\Hk,!}f^*_{\Hk}[\alpha]\in H^i(\Hk_G\times C) \end{equation} where $[\alpha]$ is defined in \eqref{eq:localcommutator}.
\end{thm}

The proof of Theorem \ref{thm:universalcommutator} will be given in \S\ref{sec:proofcomm}.

\begin{remark}
    Theorem \ref{thm:universalcommutator} has an equivalent formulation, Theorem \ref{thm:universalcommutatormodule}. The proof of these two equivalent theorems confirms \cite[Conjecture\,4.45]{liu2025higherperiodintegralsderivatives} in the special case that $X=\textup{point}$ and one of $V,W$ is the trivial representation. Since \cite[Assumption\,4.46]{liu2025higherperiodintegralsderivatives} is not satisfied in Theorem \ref{thm:universalcommutatormodule}, the method of proof in \cite{liu2025higherperiodintegralsderivatives} is not applicable here.
\end{remark}

\begin{remark}
    It is obvious that both sides of \eqref{eq:universalcommutator} have the same image under the map $\D_{\Hk}^*:H^i(\Hk_G\times C)\to H^i(\Hk_G)$. Indeed, we have $\D_{\Hk}^*(\lf^*\alpha-\rf^*\alpha)=f_{\Hk}^*(\lh^*_{\loc}\alpha-\rh^*_{\loc}\alpha)=f_{\Hk}^*(\hbar\cdot [\alpha])=c_1(T_C)\cdot f_{\Hk}^*[\alpha]=\D_{\Hk}^*\D_{\Hk,!}f^*_{\Hk}[\alpha]$. This verifies the compatibility between the sign on both sides of \eqref{eq:universalcommutator} and the normalization $\hbar = c_1(T_D)\in H^2(\mathbb{B}\Aut(D))$.
\end{remark}

For later use, we note the following immediate corollary of Theorem \ref{thm:universalcommutator}: 
\begin{cor}\label{cor:universalcommutator}
    For any $\alpha\in H^i(\mathbb{B}(L^+G\rtimes\Aut(D)))$ and $z\in H_*(C)$, we have \begin{equation}\label{eq:universalcommutatorev}
        \lh^*(\pr_1^{\Bun_G})^*\alpha^z-\rh^*(\pr_1^{\Bun_G})^*\alpha^z=f_{\Hk}^*[\alpha]\cup l^*\PD(z)
    \end{equation} where $\pr_1^{\Bun_G}:\Bun_G\times C\to \Bun_G$ is the projection to the first factor and $\lh,\rh,f_{\Hk}$ are the maps in \eqref{diag:localtoglo}. The map $\mathrm{PD}:H_*(C)\isom  H^{2-*}(C)$ is the isomorphism induced from the Poincar\'e duality characterized by $\int_{C}(\zeta\cup\PD(z))=\int_z \zeta$ for every $\zeta\in H^{*}(C)$.
\end{cor}
\begin{proof}
    Consider the map $\pr_1^{\Hk_G}:\Hk_G\times C\to \Hk_G$ (resp. $\pr_2^{\Hk_G}:\Hk_G\times C\to C$) given by projection to the first (resp. second) factor. Note that the Gysin map can be written as \[\D_{\Hk,!}=\D_{\Hk,!}\D_{\Hk}^*(\pr_1^{\Hk_G})^*=(\pr_1^{\Hk_G})^*(-)\cup \D_{\Hk,!}1=\pr_1^*(-)\cup (1\otimes\xi-\beta+\xi\otimes 1)\] where $\beta=\sum\zeta_i\otimes\zeta^i$ for a basis $(\zeta_i)$ of $H^1(C)$ with dual basis $(\zeta^i)$ satisfying $\zeta_i\cup\zeta^j=\delta_{ij}\xi$. The equation \eqref{eq:universalcommutator} can be rewritten as \begin{equation}\label{eq:universalcommutator2}
    \lf^*\alpha-\rf^*\alpha=(\pr_1^{\Hk_G})^*f^*_{\Hk}[\alpha] \cup (1\otimes\xi-\beta+\xi\otimes 1) .
\end{equation}   The desired identity follows from applying $\int_z=(\pr_1^{\Hk_G})_!(-\cup (\pr_2^{\Hk})^*\PD(z))$ to the identity \eqref{eq:universalcommutator2}.

\end{proof}

\subsection{Comparing local and global relative Hecke operators}
Now we apply Theorem \ref{thm:universalcommutator} to compute the global relative Hecke operators $\Gamma_v$. We fix a local volume datum $v=(\cK,\frc^{\loc},\frd^{\loc})$ and study the map $\Gamma_{v}^e$ evaluated on tautological classes.

Using \eqref{eq:universalcommutatorev}, for each $\alpha\in H^*(\mathbb{B}(L^+G\rtimes\Aut(D)))$ and $z\in H_*(C)$, we have \begin{equation}\label{eq:gammaformula1}\begin{split}
    \Gamma_{v}(\alpha^z)&=(\pr_{1}^{\Bun_G})_!H^*(\frd)H^*(\frc)(\pr_{1}^{\Bun_G})^*\alpha^z \\ &=(\pr_{1}^{\Bun_G})_!H^*(\frd)(\lh^*(\pr_{1}^{\Bun_G})^*\alpha^z\cdot H^*(\frc)1) \\ &=(\pr_{1}^{\Bun_G})_!H^*(\frd)((\rh^*(\pr_{1}^{\Bun_G})^*\alpha^z+f_{\Hk}^*[\alpha]\cup l^* \PD(z))\cdot H^*(\frc)1) \\ 
    &=(\pr_{1}^{\Bun_G})_!((\pr_{1}^{\Bun_G})^*\alpha^z\cdot H^*(\frd)H^*(\frc)1)+ (\pr_{1}^{\Bun_G})_!H^*(\frd)(f_{\Hk}^*[\alpha]\cup l^*\PD(z)\cdot H^*(\frc)1) \\ &=\alpha^z\cdot (\pr_{1}^{\Bun_G})_!(f^*H^*(\frd^{\loc})H^*(\frc^{\loc})1)+(\pr_{1}^{\Bun_G})_!(f^*H^*(\frd^{\loc})([\alpha]\cdot H^*(\frc^{\loc})1)\cup (\pr^{\Bun_G}_2)^*\PD(z))
    \\ &=\alpha^z\cdot \Gamma_{v}^{\loc}(1)^{\xi^*}+ \widetilde{\nabla}_{v}^{\loc}(\alpha)^z
    \\ &=(2g-2)b_v\cdot \alpha^z+ \widetilde{\nabla}_{v}^{\loc}(\alpha)^z
\end{split}\end{equation} where the map $\Gamma_v^{\loc}$ is defined in \eqref{eq:localhmap}, the map $\widetilde{\nabla}_v^{\loc}$ is defined in \eqref{eq:nablatildeloc}, and $\xi^*\in H_2(C)$ is the fundamental class of the curve. This is the global counterpart of \eqref{eq:gammaformulaloc1}.

Similarly, for $\alpha_i\in H^*(\mathbb{B}(L^+G\rtimes\Aut(D))), z_i\in H_*(C), i\in I$, we have \begin{equation}\label{eq:gammaformula}
    \begin{split}
        \Gamma_{v}(\prod_{i\in I}\alpha_i^{z_i})=\sum_{J\sub I}\pm (\prod_{j\in J}\alpha_j^{z_j})\pr_{1,!}(f^*H^*(\frd^{\loc})(\prod_{i\in I\backslash J}[\alpha_i]\cdot H^*(\frc^{\loc})1)\cup \prod_{i\in I\backslash J}\PD(z_i)) \\ =\sum_{J\sub I}\pm(\prod_{j\in J}\alpha_j^{z_j}) \Gamma_{({\prod_{i\in I\backslash J}[\alpha_i]})\cdot v}^{\loc}(1)^{\cap_{i\in I\backslash J}z_i}
    .\end{split}
\end{equation} Here, we define $z_i\cap z_j=\PD^{-1}(\PD(z_i)\cup \PD(z_j))$.

This formula is parallel to \cite[(5.5.13)]{FYZvolume} and is a global analogue of \eqref{eq:gammaformulaloc}.


\subsection{Global relative Hecke operators on associated graded}\label{sec:grgamma}

We have the following global analogue of Lemma \ref{lem:localder}:
\begin{lemma}\label{lem:globalder}
    The map $\Gamma_v^e:H^*(\Bun_G^{e+e_{\cK}})\to H^*(\Bun_G^e)$ preserves the Ran filtration $F^{\Ran}_{\bullet}$ defined on both sides in \S\ref{sec:Ranfil}. Moreover, the induced map on associated graded\footnote{The local analogue of taking associated graded with respect to the Ran filtration should be regarded as taking $-\otimes_{k[\hbar]}k$.} $\Gr^{\Ran}_{\bullet}\Gamma_v^e:\Gr^{\Ran}_{\bullet}H^*(\Bun_G^{e+e_{\cK}})\to \Gr^{\Ran}_{\bullet}H^*(\Bun_G^e)$ is a differential operator of order at most 1, which preserves the augmentation filtration $F^{\bl}_{\aug}$ on both sides.
\end{lemma}

\begin{proof}
The first claim follows from $|z_i\cap z_j|\leq \min\{|z_i|,|z_j|\}$ for any $z_i, z_j\in H_{\bullet}(C)$. The second claim follows from the observation that only the terms in \eqref{eq:gammaformula} satisfying $|J|\geq |I|-1$ survive under $\Gr_{\bullet}^{\Ran}\Gamma_{v}^e$. The last claim, concerning the preservation of the augmentation filtration, follows from \eqref{eq:gammaformula1} and the Atiyah--Bott formula (Theorem \ref{thm:AB}).
\end{proof}

By Lemma \ref{lem:globalder}, we get a map \[\Gr_{\aug}^{\bl}\Gr^{\Ran}_{\bullet}\Gamma_v^e:\Gr_{\aug}^{\bl}\Gr^{\Ran}_{\bullet}H^*(\Bun_G^{e+e_{\cK}})\to \Gr_{\aug}^{\bl}\Gr^{\Ran}_{\bullet}H^*(\Bun_G^e).\] By Theorem \ref{thm:AB}, we have $\Gr_{\aug}^{\bl}\Gr^{\Ran}_{\bullet}H^*(\Bun_G^{e+e_{\cK}})\cong \Sym^{\bl}(\mathbb{V}\otimes H_{\bullet}(C))$. Under this isomorphism, we can identify the map above as
\begin{equation}
    \Gr_{\aug}^{\bl}\Gr^{\Ran}_{\bullet}\Gamma_v^e:\Sym^{\bl}(\mathbb{V}\otimes H_{\bullet}(C))\to \Sym^{\bl}(\mathbb{V}\otimes H_{\bullet}(C)).
\end{equation} 

On the other hand, the local relative Hecke operator on the Gross motive \eqref{eq:tmap} induces a derivation
\[
\nabla_{E_v\otimes \id}:\Sym^{\bl}(\mathbb{V}\otimes H_{\bullet}(C))\to \Sym^{\bl}(\mathbb{V}\otimes H_{\bullet}(C)),
\]
defined by requiring that its restriction to generators is given by $\nabla_{E_v\otimes \id}|_{\mathbb{V}\otimes H_{\bullet}(C)}=E_v\otimes\id$, and extended to the symmetric algebra via the Leibniz rule. We have the following key observation, which is a global counterpart of Proposition \ref{prop:grgammaformulaloc}, generalizing \cite[Proposition\,5.6.17]{FYZvolume}:

\begin{prop}\label{prop:grgammaformula}
    We have $\Gr_{\aug}^{\bl}\Gr^{\Ran}_{\bullet}\Gamma_v^e=(2g-2)b_v\cdot \id + \nabla_{E_v\otimes \id}\in \End(\Sym^{\bl}(\mathbb{V}\otimes H_{\bullet}(C)))$.
\end{prop}

\begin{proof}
    This is immediate from \eqref{eq:gammaformula} and \eqref{eq:gammaformula1}.
\end{proof}

\begin{proof}[Proof of Theorem \ref{thm:main}]
    The proof given in \cite[\S5.6.18]{FYZvolume} carries over verbatim after replacing \cite[Proposition\,5.6.17]{FYZvolume} by Proposition \ref{prop:grgammaformula}.
\end{proof}

\subsection{Proof of Theorem \ref{thm:universalcommutator}}\label{sec:proofcomm}
In this section, we prove the key identity \eqref{eq:universalcommutator}.
\subsubsection{Multiplicativity in $\alpha$}
We first prove that \eqref{eq:universalcommutator} holds for $\alpha_1\alpha_2$ if it holds for $\alpha_1$ and $\alpha_2$. Indeed, assume \eqref{eq:universalcommutator} holds for $\alpha_1$ and $\alpha_2$. We have \[\begin{split}
    \lf^*(\alpha_1\alpha_2)-\rf^*(\alpha_1\alpha_2) &=\lf^*\alpha_1(\lf^*\alpha_2-\rf^*\alpha_2)+(\lf^*\alpha_1-\rf^*\alpha_1)\rf^*\alpha_2 \\
    &=\lf^*\alpha_1\D_{\Hk,!}f_{\Hk}^*[\alpha_2]+\D_{\Hk,!}f_{\Hk}^*[\alpha_1]\rf^*\alpha_2\\
    &=\D_{\Hk,!}(\lh^*f^*\alpha_1f_{\Hk}^*[\alpha_2]+f_{\Hk}^*[\alpha_1]\rh^*f^*\alpha_2) \\
    &=\D_{\Hk,!}f_{\Hk}^*(\lh_{\loc}^*\alpha_1[\alpha_2]+[\alpha_1]\rh_{\loc}^*\alpha_2)\\
    &=\D_{\Hk,!}f_{\Hk}^*[\alpha_1\alpha_2]
\end{split}.\] This verifies \eqref{eq:universalcommutator} for $\alpha_1\alpha_2$.

\subsubsection{Functoriality for $G_1\to G_2$}
Assume we have a group homomorphism between split reductive groups $G_1\to G_2$, which induces maps $\phi_{\loc}:\mathbb{B}(L^+G_1\rtimes\Aut(D))\to\mathbb{B}(L^+G_2\rtimes\Aut(D))$, $\phi:\Bun_{G_1}\to\Bun_{G_2}$, $\phi_{\loc,\Hk}:\Hk_{G_1}^{\loc}\to\Hk_{G_2}^{\loc}$, and $\phi_{\Hk}:\Hk_{G_1}\to\Hk_{G_2}$. We claim that \eqref{eq:universalcommutator} for $\alpha\in H^*(\mathbb{B}(L^+G_2\rtimes\Aut(D)))$ implies \eqref{eq:universalcommutator} for $\phi_{\loc}^*\alpha\in H^*(\mathbb{B}(L^+G_1\rtimes\Aut(D)))$. Indeed, assume \eqref{eq:universalcommutator} holds for $\alpha\in H^*(\mathbb{B}(L^+G_2\rtimes\Aut(D)))$. We have \[\begin{split}
    \lf^*(\phi_{\loc}^*\alpha)-\rf^*(\phi_{\loc}^*\alpha) &= (\phi_{\Hk}\times\id_C)^*(\lf^*\alpha-\rf^*\alpha) \\
    &=(\phi_{\Hk}\times\id_C)^*\D_{\Hk,!}f_{\Hk}^*[\alpha] \\
    &=\D_{\Hk,!}f_{\Hk}^*\phi_{\loc,\Hk}^*[\alpha] \\
    &=\D_{\Hk,!}f_{\Hk}^*[\phi_{\loc}^*\alpha]
\end{split}.\] This verifies \eqref{eq:universalcommutator} for $\phi_{\loc}^*\alpha$.

\subsubsection{Equivalent formulation}
Note that Theorem \ref{thm:universalcommutator} is equivalent to the following:
\begin{thm}\label{thm:universalcommutatormodule}
    For each $\cK\in\Shv(\Hk_G^{\loc})^{\heartsuit}$, $m\in H^*(\cK)$, and $\alpha\in H^*(\mathbb{B}(L^+G\rtimes\Aut(D)))$, we have \begin{equation}\label{eq:universalcommutatormodule}
       \lf^*\alpha\pr_1^*f_{\Hk}^* m-\rf^*\alpha\pr_1^*f_{\Hk}^* m=\D_{\Hk,!}f^*_{\Hk}([\alpha]m)\in H^*(\Hk_G\times C,\pr_1^*f_{\Hk}^*\cK)
    \end{equation} where $\pr_1:\Hk_G\times C\to \Hk_G$ is the projection to the first factor and $f_{\Hk}^*:H^*(\Hk^{\loc}_G,\cK)\to H^*(\Hk_G,f_{\Hk}^*\cK)$ is the pullback map.
\end{thm}

To see the equivalence, in one direction, since $\D_{\Hk,!}f^*_{\Hk}([\alpha]m)=\D_{\Hk,!}f^*_{\Hk}([\alpha])\cdot \pr_1^* f_{\Hk}^*m$, we know that \eqref{eq:universalcommutatormodule} can be obtained from the identity \eqref{eq:universalcommutator} by acting on $\pr_1^*f_{\Hk}^* m$. Therefore, Theorem \ref{thm:universalcommutator} implies Theorem \ref{thm:universalcommutatormodule}. 

Conversely, assume Theorem \ref{thm:universalcommutatormodule} is true. To see \eqref{eq:universalcommutator}, we only need to show that any element $\alpha\in H^*(\Hk_G\times C)$ annihilating $\pr_1^*f_{\Hk}^*m\in I\!H^*(\Hk_{G,\leq\lambda}\times C)$ for all $\lambda\in X_*(T)_+$ and $m\in I\!H^*(\Hk_{G,\leq\lambda}^{\loc})$ is zero. Since \[H^*(\Hk_G\times C)=\lim_{\lambda} H^*(\Hk_{G,\leq\lambda}\times C),\] we only need to show that for any $\lambda\in X_*(T)_+$, any element $\alpha\in H^*(\Hk_{G,\leq\lambda}\times C)$ annihilating \[\pr_1^*f_{\Hk}^* I\!H^*(\Hk_{G,\leq\lambda}^{\loc})\sub I\!H^*(\Hk_{G,\leq\lambda}\times C)\] is zero. By Künneth formula, we are reduced to showing that the annihilator of \[f_{\Hk}^*I\!H^*(\Hk_{G,\leq\lambda}^{\loc}) \sub I\!H^*(\Hk_{G,\leq\lambda})\] in $H^*(\Hk_{G,\leq\lambda})$ is zero. Note that we have \[H^*(\Hk_{G,\leq\lambda})=H^*(\Bun_G\times C)\otimes_{R^W[\hbar]} H^*(\Hk_{G,\leq\lambda}^{\loc})\] and \[I\!H^*(\Hk_{G,\leq\lambda})=H^*(\Bun_G\times C)\otimes_{R^W[\hbar]} I\!H^*(\Hk_{G,\leq\lambda}^{\loc}).\]  Since the map $H^*(\Hk_{G,\leq\lambda}^{\loc})\to \End_{R^W[\hbar]}(I\!H^*(\Hk_{G,\leq\lambda}^{\loc}))$ is injective, we only need to show that \[Q:=\coker(H^*(\Hk_{G,\leq\lambda}^{\loc})\to \End_{R^W[\hbar]}(I\!H^*(\Hk_{G,\leq\lambda}^{\loc})))\] is a free $R^W[\hbar]$-module. By \cite[Theorem\,1.5]{ginzburg1998loopgrassmanniancohomologyprincipal}, we know \[[H^*(\Hk_{G,\leq\lambda}^{\loc})\to \End_{R^W[\hbar]}(I\!H^*(\Hk_{G,\leq\lambda}^{\loc}))]\otimes_{R^W[\hbar]}k=[H^*(\Gr_{G,\leq \lambda})\to \End(I\!H^*(\Gr_{G,\leq\lambda}))]\] is an injection. Since both $\End_{R^W[\hbar]}(I\!H^*(\Hk_{G,\leq\lambda}^{\loc}))$ and $H^*(\Hk_{G,\leq\lambda}^{\loc})$ are free $R^W[\hbar]$-modules, we know that $\mathrm{Tor}^{\leq -1}_{R^W[\hbar]}(Q,k)=0$. Since $Q$ is a finitely generated graded $R^W[\hbar]$-module, this implies that $Q$ is a free $R^W[\hbar]$-module.

\subsubsection{Propagation through convolution}
We claim that Theorem \ref{thm:universalcommutatormodule} is stable under the convolution of $\Shv(\Hk_{G}^{\loc})^{\heartsuit}$, that is, assuming \eqref{eq:universalcommutatormodule} holds for $m_1\in H^*(\cK_1)$ and $m_2\in H^*(\cK_2)$ for $\cK_1,\cK_2\in\Shv(\Hk_{G}^{\loc})^{\heartsuit}$, one can deduce that \eqref{eq:universalcommutatormodule} holds for $m_1\otimes m_2\in H^*(\cK_1 * \cK_2)\cong H^*(\cK_1)\otimes_{R^W[\hbar]} H^*(\cK_2)$.

Consider the global iterated Hecke stack $\widetilde{\Hk}_{G}$ defined by the Cartesian diagram \[\begin{tikzcd}
    \widetilde{\Hk}_{G} \ar[r, "\rp"] \ar[d, "\lp"] & \Hk_G \ar[d, "\lh"] \\
    \Hk_G \ar[r, "\rh"] & \Bun_G\times C
\end{tikzcd}.\] We have three maps $\lh_2,h_2,\rh_2:\widetilde{\Hk}_G\to\Bun_G\times C$ defined by $\rh_2:=\rh\circ\rp$, $\lh_2:=\lh\circ\lp$, and $h_2:=\lh\circ\rp=\rh\circ\lp$. We also have a map $m:\widetilde{\Hk}_{G}\to \Hk_G$ given by composition of modifications. We have a map $\D_{\widetilde{\Hk}}=(\id,l\circ\rp):\widetilde{\Hk}\to\widetilde{\Hk}\times C$.

Similarly, we have the local iterated Hecke stack defined by the Cartesian diagram \[\begin{tikzcd}
    \widetilde{\Hk}_{G}^{\loc} \ar[r, "\rp_{\loc}"] \ar[d, "\lp_{\loc}"] & \Hk_G^{\loc} \ar[d, "\lh_{\loc}"] \\
    \Hk_G^{\loc} \ar[r, "\rh_{\loc}"] & {\mathbb{B}(L^+G\rtimes\Aut(D))}
\end{tikzcd}\] and maps $\lh_{\loc,2},h_{\loc,2},\rh_{\loc,2}:\widetilde{\Hk}_{G}^{\loc}\to \mathbb{B}(L^+G\rtimes\Aut(D))$, $m_{\loc}:\widetilde{\Hk}_{G}^{\loc}\to \Hk_G^{\loc}$. We have a map $f_{\Hk,2}:\widetilde{\Hk}_{G}\to\widetilde{\Hk}_{G}^{\loc}$. We have $\cK_1 *\cK_2=m_{\loc,!}(\lp_{\loc}^*\cK_1\otimes \rp_{\loc}^*\cK_2)$. 

We would like to prove \[\begin{split}
& \lf^*\alpha\pr_1^*f_{\Hk}^* (m_1\otimes m_2)-\rf^*\alpha\pr_1^*f_{\Hk}^* (m_1\otimes m_2) \\ =&\D_{\Hk,!}f^*_{\Hk}([\alpha](m_1\otimes m_2))\in H^*(\Hk_G\times C,\pr_1^*f_{\Hk}^*(\cK_1*\cK_2))
\end{split}.\] We have an isomorphism $r_{m}:H^*(\Hk_G\times C,\pr_1^*f_{\Hk}^*(\cK_1*\cK_2))\cong H^*(\widetilde{\Hk}_G\times C,\widetilde{\pr}_1^*f_{\Hk,2}^*(\lp_{\loc}^*\cK_1\otimes\rp_{\loc}^*\cK_2))$ where $\widetilde{\pr}_1:\widetilde{\Hk}_G\times C\to \widetilde{\Hk}_G$ is the projection map to the first factor. The identity above is equivalent to \[r_{m}(\lf^*\alpha\pr_1^*f_{\Hk}^*( m_1\otimes m_2)-\rf^*\alpha\pr_1^*f_{\Hk}^* (m_1\otimes m_2))=r_{m}\D_{\Hk,!}f^*_{\Hk}([\alpha](m_1\otimes m_2)).\]

Note that \[\begin{split}
    &r_{m}(\lf^*\alpha\pr_1^*f_{\Hk}^* (m_1\otimes m_2)-\rf^*\alpha\pr_1^*f_{\Hk}^* (m_1\otimes m_2)) \\
    =&(\lp\times\id)^*(\lf^*\alpha\pr_1^*f_{\Hk}^* m_1-\rf^*\alpha\pr_1^*f_{\Hk}^* m_1)\cup (\rp\times\id)^* \pr_1^*f_{\Hk}^*m_2 \\ &+  (\lp\times\id)^* \pr_1^*f_{\Hk}^*m_1 \cup (\rp\times\id)^*(\lf^*\alpha\pr_1^*f_{\Hk}^* m_2-\rf^*\alpha\pr_1^*f_{\Hk}^* m_2) \\
    =&(\lp\times\id)^*\D_{\Hk,!}f^*_{\Hk}([\alpha]m_1)\cup (\rp\times\id)^* \pr_1^*f_{\Hk}^*m_2 +  (\lp\times\id)^* \pr_1^*f_{\Hk}^*m_1 \cup (\rp\times\id)^*\D_{\Hk,!}f^*_{\Hk}([\alpha]m_2) \\
    =&\D_{\widetilde{\Hk},!}(\lp^*f_{\Hk}^*([\alpha]m_1)\cup \rp^*f_{\Hk}^* m_2+ \lp^*f_{\Hk}^*m_1\cup \rp^*f_{\Hk}^*([\alpha]m_2) ) \\
    =&\D_{\widetilde{\Hk},!}f_{\Hk,2}^*(\lp_{\loc}^*([\alpha]m_1)\cup \rp^*_{\loc} m_2 + \lp_{\loc}^*m_1\cup \rp_{\loc}^*([\alpha]m_2))\\
    =&\D_{\widetilde{\Hk},!}f_{\Hk,2}^*r_{m_{\loc}}([\alpha](m_1\otimes m_2)) \\
    =&r_{m}\D_{\Hk,!}f_{\Hk}^*([\alpha](m_1\otimes m_2))
.\end{split}\] Here, we are using the isomorphism $r_{m_{\loc}}:H^*(\Hk_G^{\loc},\cK_1*\cK_2)\cong H^*(\widetilde{\Hk}_G^{\loc},\lp_{\loc}^*\cK_1\otimes\rp_{\loc}^*\cK_2)$ and the identity $m_{\loc}^*[\alpha]=\lp_{\loc}^*[\alpha]+\rp_{\loc}^*[\alpha]$ in the next to the last step. This finishes the proof.

\subsubsection{Conclude the proof}
First, note that the images of $\phi_{\loc}^* :H^*(\mathbb{B}(L^+\GL_N\rtimes\Aut(D)))\to H^*(\mathbb{B}(L^+G\rtimes\Aut(D)))$ for all homomorphisms $G\to\GL_N$ generate $H^*(\mathbb{B}(L^+G\rtimes\Aut(D)))$ as an algebra. By the functoriality and multiplicativity, we are reduced to proving Theorem \ref{thm:universalcommutator} for $\GL_N$. By the equivalent formulation using perverse sheaves and propagation under convolution, we are reduced to showing Theorem \ref{thm:universalcommutatormodule} when $\cK=\IC_{(1,0,\cdots,0)}\in \Shv(\Hk^{\loc}_{\GL_N})^{\heartsuit}$. This case follows from \cite[Theorem\,3.2.1]{FYZvolume}, which can also be verified via a straightforward computation.

\section{Langlands dual description of eigenweights}\label{sec:langlandsdual}
In this section, we explain the local volume datum associated with the determinant line bundle $\cL_{\det}\in \Pic(\Hk_G^{\loc})_{\QQ}$ and a coweight $\lambda\in X_*(T)_+$ as promised in Example \ref{eg:localvolumedatumdet}. For simplicity, we assume that $G$ is almost simple. We give a description of the constants $b_{\lambda},\epsilon_{\lambda,i}$ involved in Theorem \ref{thm:mainsame} on the Langlands dual side.

\subsection{Local volume datum for the determinant line bundle}\label{sec:detlbloc}
We now consider a specific choice of local volume datum defined in Definition \ref{def:localvolumedatum}. 

\subsubsection{The determinant line bundle}\label{sec:detlb}

We first recall the construction of the determinant line bundle $\cL_{\det}\in \Pic(\Hk_G^{\loc})_{\QQ}$. It is normalized as follows: Let
\[
q:L^+G\bs LG \to \Hk_G^{\loc}
\]
be the natural map. We require that $q^*\cL_{\det}\in \Pic(L^+G\bs LG)\sub \Pic(L^+G\bs LG)_{\QQ}$ is ample and generates $\Pic(L^+G\bs LG^e)$ for each $e\in \pi_0(LG)$. Here, $LG^e$ is the connected component of $LG$ indexed by $e$. Moreover, $\cL_{\det}$ is relatively ample with respect to the map 
\[
\rh_{\loc}:\Hk_G^{\loc}\to \mathbb{B}(L^+G\rtimes\Aut(D))
\]
in \eqref{diag:localtoglo}.

We now recall an explicit construction of $\cL_{\det}$. For each finite-dimensional representation $V\in\Rep(G)$ with representation map $r_V:G\to \GL(V)$, we obtain an induced map 
\[
\phi_V:\Hk_G^{\loc}\to \Hk_{\GL(V)}^{\loc}.
\]
We identify $\Hk_{\GL(V)}^{\loc}$ with the moduli stack of modifications $\{\cE_1\dashrightarrow \cE_2\}$ of vector bundles of rank $\dim(V)$ over $D$ considered modulo $\Aut(D)$, and we adopt the convention that the map 
\[
\rh_{\GL(V),\loc}:\Hk_{\GL(V)}^{\loc}\to \mathbb{B}(L^+\GL(V)\rtimes\Aut(D))
\]
sends $(\cE_1\dashrightarrow \cE_2)$ to $\cE_2$. Let $\cL_{\Std}$ be the line bundle on $\Hk_{\GL(V)}^{\loc}$ whose fiber at $(\cE_1\dashrightarrow \cE_2)$ is
\[
(\cE_2:\cE_1):=\det(\cE_2/\cE_1\cap\cE_2)\otimes\det(\cE_1/\cE_1\cap \cE_2)^{-1}.
\]
Define $\cL_V:=\phi_V^*\cL_{\Std}\in \Pic(\Hk_G^{\loc})$.

Let $\Ad\in\Rep(G)$ denote the adjoint representation. The above construction yields a line bundle $\cL_{\Ad}\in \Pic(\Hk_G^{\loc})$. We then define the determinant line bundle by
\[
\cL_{\det}:=\cL_{\Ad}^{\otimes 1/(2h_G^{\vee})}\in \Pic(\Hk_G^{\loc})_{\QQ},
\]
where $h_G^{\vee}$ is the dual Coxeter number of $G$.

\subsubsection{Local volume data}
For each dominant coweight $\lambda\in X_*(T)_+$, define \begin{equation}\label{eq:dlambda} d_{\lambda}:=\langle 2\rho, \lambda\rangle.\end{equation} We use $\IC_{\l}\in \Shv(\Hk_G^{\loc})$ to denote the intersection complex of the closed Schubert cell $\Hk_{G,\leq\lambda}^{\loc}\sub \Hk_G^{\loc}$ normalized such that $\IC_{\lambda}\langle -d_{\lambda}\rangle|_{\Hk_{G,\lambda}^{\loc}}$ lies in the heart of the naive $t$-structure. Under the isomorphism \begin{equation}
    \begin{split}
        \Hom^0(\lh_{\loc,!}\IC_{\lambda}\langle d_{\lambda}\rangle, \uk_{\mathbb{B}(L^+G\rtimes\Aut(D))})\cong \Hom^0(\lh_{\loc,!}\uk_{\Hk_{G,\leq\lambda}^{\loc}}\langle 2d_{\lambda}\rangle,\uk_{\mathbb{B}(L^+G\rtimes\Aut(D))})\cong k
    \end{split}
,\end{equation} we use $[\Hk_{G,\leq\lambda}^{\loc}]^{BM}\in \Hom^0(\lh_{\loc,!}\IC_{\lambda}\langle d_{\lambda}\rangle, \uk_{\mathbb{B}(L^+G\rtimes\Aut(D))})$ to denote the element corresponding to $1\in k$, which is the \emph{(relative) fundamental (Borel-Moore homology) class} of $\Hk_{G,\l}^{\loc}$ (over $\lh_{\loc}$).

Consider the first Chern class $c_1(\cL_{\det})\in H^2(\Hk_{G}^{\loc})$. We define \[\frc_{\lambda,\det}^{\loc}:= c_1(\cL_{\det})^{d_{\lambda}+1}\cdot [\Hk_{G,\leq\lambda}^{\loc}]^{BM}   \in \Cor_{\Hk_{G}^{\loc},\IC_{\lambda}\langle -d_{\lambda} - 2\rangle}(\uk_{\mathbb{B}(L^+G\rtimes\Aut(D))}, \uk_{\mathbb{B}(L^+G\rtimes\Aut(D))}) .\] In other words, the element \[\frc_{\lambda,\det}^{\loc}\in \Hom^0(\lh_{\loc,!}\IC_{\lambda}\langle -d_{\lambda} - 2\rangle, \uk_{\mathbb{B}(L^+G\rtimes\Aut(D))})\] is given by the composition \[\begin{split}
    \lh_{\loc,!}\IC_{\lambda}\langle -d_{\lambda} - 2\rangle \xrightarrow{\lh_{\loc,!}(-\cup c_1(\cL_{\det})^{d_{\lambda}+1})} \lh_{\loc,!}\IC_{\lambda}\langle d_{\lambda}\rangle \xrightarrow{[\Hk_{G,\leq\lambda}^{\loc}]^{BM}} \uk_{\mathbb{B}(L^+G\rtimes\Aut(D))}
\end{split}.\]

Consider also the (relative) fundamental (cohomology) class (over $\rh_{\loc}$) \[\frd_{\lambda}^{\loc}:=[\Hk_{G,\leq\lambda}^{\loc}]\in \Hom^0(\uk_{\mathbb{B}(L^+G\rtimes\Aut(D))}, \rh_{\loc,!}\IC_{\lambda}\langle -d_{\lambda}\rangle),\] that is, the element corresponds to $1\in k$ under the isomorphism \[\Hom^0(\uk_{\mathbb{B}(L^+G\rtimes\Aut(D))}, \rh_{\loc,!}\IC_{\lambda}\langle -d_{\lambda}\rangle)\cong \Hom^0(\uk_{\mathbb{B}(L^+G\rtimes\Aut(D))},\rh_{\loc,*}\uk_{\Hk_{G,\leq\lambda}^{\loc}})\cong k.\]

\begin{defn}\label{def:localvolumedatumdet}
    The \emph{local volume datum given by the determinant line bundle} is the triple \[v_{\l,\det}=(\IC_{\l}, \frc^{\loc}_{\l,\det},\frd^{\loc}_{\l})\] in which each term is defined as above.
\end{defn}

\begin{remark}\label{rmk:compareFYZlb}
    Continuing with the setup in \S\ref{sec: comparision with FYZ}, for the choice in Definition \ref{def:localvolumedatumdet}, we have $\eta = \pm t_{\lambda}^{d_{\lambda}+1}$ in the notation at the beginning of \cite[\S6]{FYZvolume}.\footnote{They use $\mu$ instead of $\lambda$.}
    We emphasize that, in our convention, the eigenweights $\epsilon_{\lambda,i}$ are positive in all computed examples (although we do not prove this in general), in contrast to the normalization in \cite[\S6]{FYZvolume}, where their signs are not fixed.
    We also obtain a canonical choice of $\eta'$, which is not addressed in \cite[\S6]{FYZvolume} and gives rise to an interesting new constant $b_{\lambda}$ in Theorem \ref{thm:mainsame}.
\end{remark}

\subsection{Translation to the spectral side}\label{sec:transtospec}
In this section, we describe the constants $b_{\lambda}$ and $\epsilon_{\lambda,i}$ appearing in Theorem \ref{thm:mainsame} in terms of the Langlands dual group. 

\subsubsection{Notations}
Let $\kappa_{\min}:\frgc\times\frgc\to k$ be the $\Gc$-invariant non-degenerate bilinear form normalized such that $\kappa_{\min}(\alpha_{\mathrm{long}},\alpha_{\mathrm{long}})=2$ where $\alpha_{\mathrm{long}}\in X_*(\Tc)$ is any long coroot of $\Gc$. 

For each finite-dimensional representation $V\in \Rep(\Gc)$, let $r_V:\Gc\to \GL(V)$ be the representation map and $dr_V:\frgc\to \mathfrak{gl}(V)$ its differential. This defines a $\Gc$-invariant bilinear form
\begin{equation}\label{eq:kappaV}
\kappa_V:\frgc\times\frgc \to k,\qquad \kappa_V(X,Y)=\tr\bigl(dr_V(X)\,dr_V(Y)\bigr)
\end{equation}
for $X,Y\in \frgc$. Let $\Ad\in \Rep(\Gc)$ denote the adjoint representation. Then
\begin{equation}\label{eq:kappamin}
\kappa_{\min}=\frac{\kappa_{\Ad}}{2h^{\vee}_{\Gc}l_G}.
\end{equation}
Here $l_G$ is the lacing number, equal to $1$ for types $ADE$, $2$ for types $BCF$, and $3$ for type $G$, and $h_{\Gc}^{\vee}$ is the dual Coxeter number of $\frgc$.

We use $\check{\alpha}_j\in X^*(\Tc),\, 1\leq j\leq n$ to denote the simple roots of $\Gc$. By restriction, we can regard $\kappa_{\min}$ as a $W$-invariant bilinear form on $\frtc$ (hence also $\frt\cong \frtc^*$) satisfying $\kappa_{\min}(\check{\alpha}_{\mathrm{short}},\check{\alpha}_{\mathrm{short}})=2$ where $\check{\alpha}_{\mathrm{short}}\in X^*(\Tc)$ is any short root of $\Gc$. That is, the bilinear form $\kappa_{\min}$ corresponds to the basic bilinear form on $\frg$.

We use $\frgc_{\ZZ}$ to denote the Chevalley form of $\frgc$. Choose Chevalley generators $x_i\in \frgc_{\ZZ,\check{\alpha}_i},~1\leq i\leq n$ which are unique up to signs. Let $e=\sum_{i=1}^n \kappa_{\min}(\check{\alpha}_i,\check{\alpha}_i)x_i/2\in \frgc_{\ZZ}$. Let $\rho\in X_*(\Tc)_{\QQ} \sub \frtc$ be the half sum of all positive coroots of $\Gc$. This determines an $\mathfrak{sl}_2$-triple $(e,2\rho,f)$ for a unique $f\in \frgc$.

Using the non-degenerate bilinear form $\kappa_{\min}$, we make identification $\kappa: \frgc\cong\frgc^*$. Consider the Kostant slice $\frsc=e+\frgc_f\sub\frgc$. The adjoint action of $\rho$ induces a canonical grading $\frgc_f=\bigoplus_d\frgc_{f,d}$, where $[\rho,x]=dx$ for every $x\in\frgc_{f,d}$.
Set $\delta_G=1$ if $G$ is of type $D_n$ with $n$ even, and $\delta_G=0$ otherwise. Then $\frgc_f=\bigoplus_{j=1}^{n-\delta_G}\frgc_{f,-d_j+1}$, where $d_1<\cdots<d_{n-\delta_G}$ are the distinct degrees of the basic invariant polynomials of $\frgc$. 
Each summand has dimension one, except when $\delta_G=1$ and $j=n/2$, in which case it has dimension two.

The bilinear form $\kappa_{\min}$ restricts to a non-degenerate pairing $\kappa_{\min}:\frgc_f\times\frgc_e\to k$. This pairing induces an identification $\frgc_e\cong\frgc_f^*$. Via the identification $\frsc=e+\frgc_f\cong\frgc_f$, we obtain a natural inclusion $\frgc_e\hookrightarrow\cO(\frsc)$. Let $\frsc^*$ denote its image. Note that we also have a grading $\frsc^* \cong \frgc_e=\bigoplus_d \frgc_{e,d}$ such that we have $(\frgc_{e,-d})^*\cong \frgc_{f,d}$.

Let $\frcc=\frgc^*\sslash\Gc\cong\frgc\sslash\Gc$ be the Chevalley quotient, and let $\chi:\frgc^*\to\frcc$ be the quotient map. It induces an isomorphism  $\chi\circ \kappa|_{\frsc}:\frsc\xrightarrow[]{\sim}\frcc$ . Moreover, after a suitable normalization of the grading, the resulting pullback isomorphism $\cO(\frcc)\xrightarrow{\sim}\cO(\frsc)$ induces an isomorphism of graded vector spaces $\mathbb{V}\xrightarrow{\sim}\frsc^*$.



Let $J_{\Gc}\to \frsc$ be the regular centralizer whose fiber at $X\in \frsc$ is the centralizer subgroup $C_{\Gc}(X)\sub \Gc$. The map $\chi$ induces a map $d\chi:\frgc\times_{\frcc} T^*\frcc\to T^*\frgc\cong \frgc\times \frgc^*\cong \frgc\times\frgc$ which restricts to an isomorphism $d\chi|_{\frsc}:T^*\frsc\isom \Lie J_{\Gc}$.

We use $V_{\lambda}\in \Rep(\Gc)$ to denote the irreducible representation of $\Gc$ with highest weight $\lambda\in X_*(T)_+$. We use $V_{\lambda,\ZZ}\sub V_{\lambda}$ to denote the standard representation of $\Gc_{\ZZ}$ with highest weight $\lambda$, which is a $\ZZ$-lattice inside $V_{\lambda}$. We use $u_{\lambda,-}\in V_{\lambda,\ZZ}$ to denote a generator of the lowest weight submodule, and $u_{\lambda,+}^*\in V_{\lambda,\ZZ}^*$ to denote a generator of the lowest weight submodule of $V_{\lambda,\ZZ}^*=\Hom_{\ZZ}(V_{\lambda,\ZZ},\ZZ)$. The generators $u_{\lambda,-}$ and $u_{\lambda,+}^*$ are normalized such that $\langle e^{d_{\lambda}}\cdot u_{\lambda,-},u_{\lambda,+}^*\rangle \geq 0$. This determines the pair $u_{\lambda,-}, u_{\lambda,+}^*$ up to a simultaneous sign change. Let $V_{\lambda}=\bigoplus_{\mu\in X^*(\Tc)}V_{\lambda}(\mu)$ be the weight space decomposition.

\subsubsection{Spectral description of $\epsilon_{\lambda,j}$}\label{sec:Espectral}
Now we describe the eigenweights $\epsilon_{\lambda,j}$ involved in Theorem \ref{thm:mainsame} (see also Example \ref{eg:Edet}) as some invariants attached to the representation $V_{\lambda}\in \Rep(\Gc)$.

Consider the bilinear form $\kappa_{\lambda}:\frgc_f\times\frgc_e\to k$ defined by \begin{equation}
    \kappa_{\lambda}(X,Y)=\langle \sum_{s=0}^{d_{\lambda}}  e^sXYe^{d_{\lambda}-s}\cdot u_{\lambda,-},u_{\lambda,+}^* \rangle.
\end{equation} This determines an endomorphism $E_{\lambda}\in\End(\frgc_e)= \End(\frsc^*)$ defined such that $\kappa_{\min}(X,E_{\lambda}(Y))=\kappa_{\lambda}(X,Y)$ for any $X\in \frgc_f$ and $Y\in \frgc_e$.

\begin{thm}\label{thm:Espectral}
    Under the natural identification $\mathbb{V}\cong \frsc^*$, we have $E_{v_{\lambda,\det}}=-E_{\lambda}$.
\end{thm}

\begin{proof}
    During this proof, we exchange the role of $\lh$ and $\rh$ everywhere to match the convention in most literature, like \cite{BF}\cite{yun2011integral}. In particular, we change the definition of $\cL_{\det}$ such that it is relatively ample over $\lh_{\loc}$. Since $E_{\lambda}=E_{-w_0(\lambda)}$, we can still work with $\lambda$ instead of $-w_0(\lambda)$.

    By \cite[Theorem\,1(b)]{BF}, there is a canonical graded isomorphism $H^*(\Hk_G^{\loc})\cong \bigoplus_{z\in Z(\Gc)} \cO(D_{\frcc}(\frcc\times \frcc))$.\footnote{By $\pi_0(\Hk_G^{\loc})\cong\pi_1(G)\cong X^*(Z(\Gc))$, we get a connected component decomposition $\Hk_G^{\loc}=\coprod_{\chi\in X^*(Z(\Gc))}\Hk_G^{\loc,\chi}$. Note that different $H^*(\Hk_{G}^{\loc,\chi})$ for $\chi \in X^*(Z(\Gc))$ are canonically isomorphic. This gives a canonical diagonal embedding $H^*(\Hk_{G}^{\loc,\triv})\sub H^*(\Hk_{G}^{\loc})$ for $\triv\in X^*(Z(\Gc))$. Under the normalization we choose, this embedding corresponds to $\cO(D_{\frcc}(\frcc\times \frcc))\sub \bigoplus_{z\in Z(\Gc)}\cO(D_{\frcc}(\frcc\times\frcc))$ which is the inclusion via the direct summand indexed by the identity element $e\in Z(\Gc)$.} 
    Here, we use $D_{\frcc}(\frcc\times \frcc)$ to denote the deformation to the normal cone of the scheme $\frcc\times\frcc$ along the closed subscheme $\frcc\sub \frcc\times\frcc$ embedding along the diagonal. Under this identification, the natural maps $\lh_{\loc}^*,\rh_{\loc}^*:H^*(\mathbb{B}(L^+G\rtimes\Aut(D)))\to H^*(\Hk_G^{\loc}) $ get identified with pull-back maps for the natural projections $p_1,p_2:D_{\frcc}(\frcc\times \frcc)\to \frcc\times\AA^1$ where we identify $\AA^1\cong\Spec k[\hbar]$.

    
    The canonical isomorphism $H^*(\Hk_G^{\loc})\otimes_{k[\hbar]}k\cong H^*(L^+G\bs LG/L^+G)$ gets identified with the isomorphism $\bigoplus_{z\in Z(\Gc)}\cO(D_{\frcc}(\frcc\times \frcc)\times_{\AA^1}\{0\})\cong \bigoplus_{z\in Z(\Gc)}\cO(T\frcc)\cong \bigoplus_{z\in Z(\Gc)} U(\Lie J_{\Gc})$. Under this identification, for $x\in \cO(\frcc)$, the image of $[x]=(\rh_{\loc}^*x-\lh_{\loc}^*x)/\hbar$ in $H^*(L^+G\bs LG/L^+G)$ is the differential $1$-form $-dx \in \Gamma(\frcc,\Omega_{\frcc})\sub \cO(T\frcc)\sub \bigoplus_{z\in Z(\Gc)}\cO(T\frcc)$ where the last inclusion is via the direct summand indexed by the identity element $e\in Z(\Gc)$.\footnote{To see this subtle minus sign here, consider the filtration $F_{\geq \mu}I\!H^*(T\bs \Gr_{G,\leq\lambda}/\Gm)\sub I\!H^*(T\bs \Gr_{G,\leq\lambda}/\Gm)$ as in the proof of Theorem \ref{thm:bspectral}. One can check that taking cup product with the image of $[x]$ in $H^*(T\bs \Gr_G/\Gm)$ coincides with the multiplication by $-\partial_{\mu}x\in \cO(\frtc)$ on the graded piece $\Gr_{\mu}^FI\!H^*(T\bs \Gr_{G,\leq\lambda}/\Gm)$ where $\partial_{\mu}x$ is the partial derivative of $x\in \cO(\frcc)\sub \cO(\frtc)$ along $\mu\in \frt\cong \frtc$.}
    The first Chern class $c_1(\cL_{\det})\in H^*(L^+G\bs LG/L^+G)$ is identified with the tautological section $e^{\univ}:=(X\in \frsc \mapsto X\in \Lie J_{\Gc})\in \Lie J_{\Gc}\sub U(\Lie J_{\Gc})$.

By \cite[Theorem\,4]{BF}, the graded $H^*(L^+G\bs LG/L^+G)$-module $I\!H^*(L^+G\bs \Gr_{G,\leq\lambda})$ is identified with the graded vector space $V_{\lambda}\otimes \cO(\frsc)$ equipped with the tautological graded $\bigoplus_{z\in Z(\Gc)}U(\Lie J_{\Gc})$-action induced by the natural embedding $\Lie J_{\Gc}\hookrightarrow\check{\mathfrak{g}}\otimes\cO(\frsc)$.
 Under this identification, we can choose the lowest weight generators properly such that $[\Hk_{G,\leq\lambda}^{\loc}]~\text{mod $\hbar$}= u_{\lambda,-}\otimes 1\in V_{\lambda}\otimes\cO(\frsc)$ and $[\Hk_{G,\leq\lambda}^{\loc}]^{BM}~\text{mod $\hbar$}=u_{\lambda,+}^*\otimes 1\in V_{\lambda}^*\otimes\cO(\frsc)$.

    Therefore, the map $E_{v_{\lambda,\det}}:\frsc^*\to \frsc^*$ is determined as follows: For $x\in  \frsc^*$, we have \begin{equation}\label{eq:geopairing} \langle(e^{\univ})^{d_{\lambda} + 1}(-dx)\cdot (u_{\lambda,-}\otimes 1),u^*_{\lambda,+}\otimes 1\rangle =E_{v_{\lambda,\det}}(x)+\textup{other terms not in $\frsc^*\sub \cO(\frsc)$}.\end{equation}

Here we regard $x$ as an element of $\cO(\frcc)$ and $dx$ as an element of $\Lie J_{\Gc}$ via the isomorphism $\Lie J_{\Gc}\cong\Gamma(\frcc,\Omega_{\frcc})$. Choose a homogeneous basis $f_1,\ldots,f_n$ of $\frgc_f$ with respect to the grading $\frgc_f=\bigoplus_{j=1}^{n-\delta_G}\frgc_{f,-d_j+1}$. Let $a_1,\ldots,a_n\in\frsc^*$ be the dual basis, characterized by $a_i(f_j)=\delta_{ij}$. The universal element is then given by $e^{\univ}=e\otimes 1+\sum_{j=1}^n f_j\otimes a_j\in\frgc\otimes\cO(\frsc)$. Since $e^{d_{\lambda}+1}$ acts by zero on $V_{\lambda}\otimes\cO(\frsc)$, from \eqref{eq:geopairing}, we get \begin{equation}
        E_{v_{\lambda},\det}(a_j)=-\sum_{i, \deg(f_i)=\deg(f_j)}\langle\sum_{s=0}^{d_{\lambda}} e^s f_i e^{d_{\lambda}-s} (da_j)|_{e\in\frsc}\cdot u_{\lambda,-},u_{\lambda,+}^* \rangle\cdot a_i
    \end{equation} 
where $(da_j)|_e\in\frgc_{e,d_{j'}-1}$ whenever $f_j\in\frgc_{f,-d_{j'}+1}$. It therefore suffices to show that $\kappa_{\min}((da_j)|_e,f_i)=\delta_{ij}$ for all $i,j$. This follows by differentiating the identity $a_j(e+\alpha f_i)=\alpha\delta_{ij}$ at $\alpha=0$.

\end{proof}

Theorem \ref{thm:Espectral} has the following immediate corollary:
\begin{cor}\label{cor:epsilongenformula}
    When $G$ is not of type $D_n$ for even $n$, we have $\epsilon_{\lambda,j}=\frac{\kappa_{\lambda}|_{\frgc_{f,-d_j+1}\otimes \frgc_{e,d_j-1}}}{\kappa_{\min}|_{\frgc_{f,-d_j+1}\otimes \frgc_{e,d_j-1}}}\in \QQ$ for $1\leq j\leq n$.
\end{cor}


\subsubsection{Spectral description of $b_{\lambda}$}\label{sec:bspectral}
In this subsection, we describe the number $b_{\lambda}\in \QQ$ involved in Theorem \ref{thm:mainsame} (see also Example \ref{eg:Edet}) as an invariant attached to the representation $V_{\lambda}\in \Rep(\Gc)$.

\begin{thm}\label{thm:bspectral}
    We have \[b_{\lambda}=\langle \sum_{s=0}^{d_{\lambda}}e^s H_2 e^{d_{\lambda}-s }u_{\lambda,-},u_{\lambda, +}^*\rangle \in \QQ \] where $H_2=(\sum_{\alpha \in \Phi_G}(d\alpha)^2)/{4h^{\vee}_G}\in U(\frtc)\sub U(\frgc)$, $h^{\vee}_G$ is the dual Coxeter number of $G$.
\end{thm}
\begin{proof}
    As we did in the proof of Theorem \ref{thm:Espectral}, we exchange the role of $\lh$ and $\rh$ everywhere.
    
    Note that we can work with $H^*(\Gr_G/\Gm)$ instead of $H^*(\Hk_G^{\loc})$ where $\Gm$ acts on $\Gr_G$ via loop rotation. By the proof of \cite[Lemma\,2.2]{yun2011integral}, there is a canonical isomorphism $I\!H^*(\Gr_{G,\leq\lambda}/\Gm)\cong V_{\lambda}\otimes k[\hbar]$. Moreover, let $T_{\mu}=LN^-t^{\mu}L^+G/L^+G\sub \Gr_{G}$ and $\overline{T_{\mu}}$ be its closure for $\mu\in X_*(T)$. Define \[F^{\geq \mu}I\!H^*(\Gr_{G,\leq\lambda}/\Gm)=\im(I\!H^*_{(\overline{T_{\mu}}\cap \Gr_{G,\leq\lambda})/\Gm}(\Gr_{G,\leq\lambda}/\Gm)\to I\!H^*(\Gr_{G,\leq\lambda}/\Gm)).\] This gives a filtration $\{F^{\geq \mu}I\!H^*(\Gr_{G,\leq\lambda}/\Gm)\sub I\!H^*(\Gr_{G,\leq\lambda}/\Gm)\}_{\mu \in X_*(T)}$. On the spectral side, this filtration corresponds to the obvious filtration by weights $\{(G^{\geq \mu}V_{\lambda})\otimes k[\hbar]\}_{\mu \in X^*(\Tc)}$. Passing to associated graded pieces, this isomorphism gives $\mathrm{Gr}^{\mu}F(I\!H^*(\Gr_{G,\leq\lambda}/\Gm))\cong i_{\mu}^*j_{\mu}^!\IC_{\Gr_{G,\leq\lambda}/\Gm}\otimes H^*(\mathbb{B}\Gm)\cong V_{\lambda}(\mu)\otimes k[\hbar]$. Here $i_{\mu}:\{t_{\mu}\}\to (T_{\mu}\cap \Gr_{G,\leq\lambda})/\Gm$ and $j_{\mu}:(T_{\mu}\cap \Gr_{G,\leq\lambda})/\Gm\to \Gr_{G,\leq\lambda}/\Gm$ are the natural maps.

    Note that taking cup product with $c_1(\cL_{\det})$ preserves the filtration $F^{\geq \mu}I\!H^*((L^+G\bs LG_{\leq\lambda})/\Gm)$. Therefore, we get an induced endomorphism $\cup c_1(\cL_{\det})$ on the associated graded pieces $\mathrm{Gr}^{\mu}F(I\!H^*(\Gr_{G,\leq\lambda}/\Gm))$, which coincides with $\cup c_1(\cL_{\det}|_{[\{t^{\mu}\}/\Gm]})$ on $i_{\mu}^*j_{\mu}^!\IC_{\Gr_{G,\leq\lambda}/\Gm}\otimes H^*(\mathbb{B}\Gm)$. Note that $[\cL_{\det}|_{[\{t^{\mu}\}/\Gm]}] \in \Pic(\mathbb{B}\Gm)_{\QQ}\cong \QQ$ can be identified with $-(\sum_{\alpha \in \Phi_{G}}\langle\alpha,\mu\rangle^2)/4h^{\vee}_G =-\kappa_{\min}(\mu,\mu)/2$.

    On the other hand, we know that $\cup c_1(\cL_{\det})$ modulo $\hbar$ can be identified with the action of $e\in \frgc$ on $V_{\lambda}$. Unwinding the construction, we know $b_{\lambda}$ is the coefficient of $\hbar$ in $-\langle c_1(\cL_{\det})^{d_{\lambda}+1}\cdot u_{\lambda,-},u_{\lambda,+}^*\rangle \in k[\hbar]$. The desired identity follows immediately.
\end{proof}

\begin{remark}\label{rmk:BZSV}
Here is another description of $E_{\lambda}$ and $b_{\lambda}$, which reveals their relation to relative Langlands duality as developed in \cite{BZSV}. We follow the notation of \cite[\S4]{liu2025higherperiodintegralsderivatives}. Let $\Mc=T^*_{\psi}(\Gc/\Nc) \cong \Gc\times\frsc$ be the $\Gc$-Hamiltonian space corresponding to the Whittaker period. It is the relative Langlands dual of the trivial $G$-Hamiltonian space $M=T^*X$ for $X=*$. The elements $\frc^{\loc}_{\lambda,\det}$ and $\frd^{\loc}_{\lambda}$ give rise to maps
$\alpha:V_{\lambda}\to \PL_{X,\hbar}$ and $\beta:V_{\lambda}^*\to \PL_{X,\hbar}$.
Reducing modulo $\hbar$ and using the isomorphism $\PL_{X,\hbar}\otimes_{k[\hbar]}k \cong \cO(\Mc)$, we obtain maps
$\overline{\alpha}:V_{\lambda}\to \cO(\Mc)$ and $\overline{\beta}:V_{\lambda}^*\to \cO(\Mc)$. Consider the canonical element $\mathrm{unit}\in V_{\lambda}\otimes V_{\lambda}^*$. Then
$m\circ (\alpha\otimes\beta)(\mathrm{unit}) = b_{\lambda}\cdot \hbar \in \PL_{X,\hbar}$,
where $m:\PL_{X,\hbar}\otimes \PL_{X,\hbar} \to \PL_{X,\hbar}$ is the multiplication map. Define a vector field
$X_{\lambda}:=\overline{m}\circ (d\otimes \id)\circ (\overline{\alpha}\otimes\overline{\beta})(\mathrm{unit}) \in \Gamma(\Mc,T^*\Mc)\cong \Gamma(\Mc,T\Mc)$,
where $d:\cO(\Mc)\to \Gamma(\Mc,T^*\Mc)$ is the de Rham differential and $\overline{m}:\cO(\Mc)\otimes \Gamma(\Mc,T^*\Mc)\to \Gamma(\Mc,T^*\Mc)$ is the multiplication map.
The vector field $X_{\lambda}$ vanishes at $(\id,e)\in \Gc\times\frsc\cong \Mc$, and hence induces a Hessian endomorphism
$H_{X_{\lambda},(\id,e)}:T_{(\id,e)}^*\Mc\to T_{(\id,e)}^*\Mc$.
Under the identification $T_{(\id,e)}^*\Mc\cong \frgc\times \frsc^*$, the semisimplification of the Hessian endomorphism is given by
$H_{X_{\lambda},(\id,e)}^{\mathrm{ss}}=0\times E_{\lambda}\in \End(\frgc\times \frsc^*)$.
This description can be further simplified by replacing $\Mc$ with its Whittaker reduction
$T^*((\Nc,\psi)\bs \Gc/(\Nc,\psi))\cong J_{\Gc}$.
It turns out that such Hessians controlling arithmetic intersection numbers of special cycles constitute a general phenomenon for hyperspherical varieties. We plan to investigate this direction in future work.
\end{remark}

\section{Examples}\label{sec:examples}
In this section, we work out the constants $\epsilon_{\lambda,i},b_{\lambda}$ appearing in Theorem \ref{thm:mainsame} explicitly in some cases. Some computations are done with the help of SageMath \cite{sagemath}. The source code and auxiliary computations used in this section are available at \url{https://github.com/zeyuw42/eigenweights}.

\subsection{Examples of \texorpdfstring{$\epsilon_{\lambda,j}$}{epsilon}}\label{sec:Eexamples}
In this section, we give examples of the number $\epsilon_{\lambda,j}\in \overline{\QQ}$ for $\lambda\in X^*(\Tc)_+$ and $1\leq j\leq n$, where $n$ is the rank of $G$. We assume $\lambda\neq 0$.

\subsubsection{Reduction to fundamental weights}
The computation of the numbers $\epsilon_{\lambda,j}$ can be reduced to the computation of invariants of fundamental weights $\varpi_i \in X^*(\Tc)_+$. In fact, the computation for $\lambda_1 +\lambda_2 \in X^*(\Tc)_+$ can be reduced to $\lambda_1,\lambda_2\in X^*(\Tc)_+$. This is provided by Proposition \ref{prop:redtofundepsilon}.

Consider $S_{\lambda}\in \frgc_e^*$ defined such that for $Y\in \frgc_{e,d_j-1}$, one has \begin{equation}
    S_{\lambda}(Y)=\langle e^{d_{\lambda}-d_j+1}Y\cdot u_{\lambda,-},u_{\lambda,+}^*\rangle.
\end{equation}
Similarly, one defines $T_{\lambda}\in \frgc_{f}^*$ such that for $X\in \frgc_{f,-d_j+1}$, one has \begin{equation}
    T_{\lambda}(X)=\langle\sum_{s=0}^{d_{\lambda}+d_j-1}e^sXe^{d_{\lambda}+d_j-1-s}\cdot u_{\lambda,-},u_{\lambda,+}^*\rangle.
\end{equation}
Recall that we introduced the notation $d_{\lambda} = \langle2\rho,\lambda\rangle\in \ZZ_{\geq 0}$. Define $\deg_{\lambda}:=\langle e^{d_{\lambda}} u_{\lambda,-} ,u_{\lambda,+}^*\rangle\in \ZZ_{\geq 0}$. We write $\kappa_{\lambda,j}:=\kappa_{\lambda}|_{\frgc_{f,-d_j+1}\otimes \frgc_{e,d_j-1}}$.

\begin{prop}\label{prop:redtofundepsilon}
    For $\lambda_1,\lambda_2\in X^*(\Tc)_+$, we have
    \begin{equation}\begin{split}\kappa_{\lambda_1+\lambda_2,j}=\binom{d_{\lambda_1}+d_{\lambda_2}+1}{d_{\lambda_1}+1}\kappa_{\lambda_1,j}\deg_{\lambda_2}+\binom{d_{\lambda_1}+d_{\lambda_2}+1}{d_{\lambda_2}+1}\kappa_{\lambda_2,j}\deg_{\lambda_1}
    \\
    +\binom{d_{\lambda_1}+d_{\lambda_2}+1}{d_{\lambda_1}-d_j+1}T_{\lambda_2,j}\otimes S_{\lambda_1,j}+\binom{d_{\lambda_1}+d_{\lambda_2}+1}{d_{\lambda_2}-d_j+1}T_{\lambda_1,j}\otimes S_{\lambda_2,j}
    \end{split}
    \end{equation}
    \begin{equation}
        S_{\lambda_1+\lambda_2,j}=\binom{d_{\lambda_1}+d_{\lambda_2}-d_j+1}{d_{\lambda_1}-d_j+1}S_{\lambda_1,j}\deg_{\lambda_2}+\binom{d_{\lambda_1}+d_{\lambda_2}-d_j+1}{d_{\lambda_2}-d_j+1}S_{\lambda_2,j}\deg_{\lambda_1}
    \end{equation}
    \begin{equation}
        T_{\lambda_1+\lambda_2,j}=\binom{d_{\lambda_1}+d_{\lambda_2}+d_j}{d_{\lambda_1}+d_j}T_{\lambda_1,j}\deg_{\lambda_2}+\binom{d_{\lambda_1}+d_{\lambda_2}+d_j}{d_{\lambda_2}+d_j}T_{\lambda_2,j}\deg_{\lambda_1}
    \end{equation}
    \begin{equation}
        \deg_{\lambda_1+\lambda_2}=\binom{d_{\lambda_1}+d_{\lambda_2}}{d_{\lambda_1}}\deg_{\lambda_1}\deg_{\lambda_2}.
    \end{equation}
\end{prop}

This proposition reduces the computation of $\kappa_{\lambda}$ for an arbitrary $\lambda\in X^*(\Tc)$ to the computation of $\kappa_{\varpi_i},S_{\varpi_i},T_{\varpi_i},\deg_{\varpi_i}$ where $\varpi_i\in X^*(\Tc)$ are the fundamental weights of $\Gc$ for $1\leq i \leq n$.

\subsubsection{Case \texorpdfstring{$j = 1$}{j=1}}
In this subsection, we derive a formula for $\epsilon_{\lambda,1}$ for any $\lambda\in X^*(\Tc)_+$. We treat this independently since the answer in this case is particularly simple.
\begin{prop}
    For any $\lambda\in X^*(\Tc)_+$, we have \begin{equation}\label{eq:epsilon1formula}
    \epsilon_{\lambda,1}=\frac{2h^{\vee}_{\Gc}l_Gd_{\lambda}(d_{\lambda} +1)(d_{\lambda}+2)\deg_{\lambda}}{\sum_{j=1}^n(2d_j-2)(2d_j-1)2d_j}.\end{equation}
\end{prop}

\begin{proof}

 Note that $\frgc_{f,-1}=\Span(f)$ and $\frgc_{e, 1}=\Span(e)$. We get \[\epsilon_{\lambda,1}=\kappa_{\lambda}(f, e)/\kappa_{\min}(f,e).\] Consider the Lie subalgebra $\mathfrak{sl}_2=\Span(e,f,2\rho)\sub  \frgc$. We have \[\Ad|_{\mathfrak{sl}_2} = \bigoplus_{j=1}^n\Sym^{2d_j-2}\Std_2.\] Note that \[\kappa_{\lambda}(f,e)=\kappa_{\SL_2,\Sym^{d_{\lambda}}\Std_2}(f,e)\cdot \deg_{\lambda}\] where $\kappa_{\SL_2,\Sym^{d_{\lambda}}\Std_2}$ is the invariant bilinear form for $\SL_2$ defined in \eqref{eq:kappaV} and \[\kappa_{\min}(f,e)=\frac{1}{2h^{\vee}_{\Gc}l_G}\kappa_{\Ad}(f,e)=\frac{1}{2h^{\vee}_{\Gc}l_G}\sum_{j=1}^n\kappa_{\SL_2,\Sym^{2d_j -2}\Std_2}(f,e).\] For any pair of integers $a,b\in \ZZ_{\geq 1}$, an easy computation shows \[\frac{\kappa_{\SL_2,\Sym^b\Std_2}}{\kappa_{\SL_2,\Sym^a\Std_2}}=\frac{b(b+1)(b+2)}{a(a+1)(a+2)}.\] This implies that\[\epsilon_{\lambda,1}=\frac{\kappa_{\SL_2,\Sym^{d_{\lambda}}\Std_2}(f,e)\deg_{\lambda}}{\kappa_{\min}(f,e)}=\frac{2h^{\vee}_{\Gc}l_Gd_{\lambda}(d_{\lambda} +1)(d_{\lambda}+2)\deg_{\lambda}}{\sum_{j=1}^n(2d_j-2)(2d_j-1)2d_j}.\]

 \end{proof}

\subsubsection{Cyclic case}
In this subsection, we illustrate a special situation where all the invariants $\epsilon_{\lambda,j}$ for different $1\leq j\leq n$ are the same. Recall we have the representation map $r_{V_{\lambda}}:\Gc\to \GL(V_{\lambda})$ whose differential is a Lie algebra homomorphism $dr_{V_{\lambda}}: \frgc\to \gl(V_{\lambda})$.

\begin{prop}\label{prop:regularcase}
    Suppose the element $dr_{V_{\lambda}}(e)\in \gl(V_{\lambda})$ is a regular nilpotent element, then we have $\epsilon_{\lambda,j}=\epsilon_{\lambda,1}$ for any $1\leq j\leq n$.
\end{prop}

\begin{proof}
    Under the assumption, we have $V_{\lambda}=\bigoplus_{s=0}^{d_{\lambda}}k\cdot e^{s}u_{\lambda,-}$. This implies that for any $Y\in \frgc_{e,d_j-1}$ and $X\in \frgc_{f,-d_j+1}$, we have \[\kappa_{\lambda}(X,Y)=\tr((dr_{V_{\lambda}})(X)\cdot (dr_{V_{\lambda}})(Y))\deg_{\lambda}=\kappa_{V_{\lambda}}(X,Y)\deg_{\lambda}\] where $\kappa_{V_{\lambda}}$ is defined in \eqref{eq:kappaV}. This shows that $\kappa_{\lambda}=\frac{\kappa_{V_{\lambda}}\deg_{\lambda}}{\kappa_{\min}}\cdot \kappa_{\min}$, which implies $\e_{\lambda,j}=\frac{\kappa_{V_{\lambda}}\deg_{\lambda}}{\kappa_{\min}}$ for any $1\leq j\leq n$.
\end{proof}
For classical groups, one can write down explicit formulas for all the invariants involved in Proposition \ref{prop:redtofundepsilon}. Compared to the formulas in \cite{feng2026eigenweightsarithmetichirzebruchproportionality}, these formulas are more elementary.

\subsubsection{Type \texorpdfstring{$A_{n}$}{A_n}}\label{sec:Aepsilon}

In this subsection, we determine all the invariants in Proposition \ref{prop:redtofundepsilon} when $G=\PGL_{n+1}$. In this case, we have $d_j = j + 1$ for $1\leq j\leq n$. We have $\Gc=\SL_{n+1}$. We order the fundamental weights as \[\dynkin[labels={\varpi_1,\varpi_2,\varpi_{n-1},\varpi_n}, edge length=.5cm]{A}{}.\] We have $V_{\varpi_k,\ZZ}=\bigwedge^k_{\ZZ}\Std_{n+1,\ZZ}\in \Rep(\SL_{n+1,\ZZ})$ where $\Std_{n+1,\ZZ}=\bigoplus_{i=1}^{n+1} \ZZ\cdot u_i$ is the standard representation. We choose the usual maximal torus of $\SL_{n+1,\ZZ}$ with respect to this basis. The invariant bilinear form is given by $\kappa_{\min}=\kappa_{\Std_{n+1}}: \mathfrak{sl}_{n+1}\times \mathfrak{sl}_{n+1}\to k$.

Take $e\in \mathfrak{sl}_{n+1,\ZZ}$ such that \[e\cdot u_i=
\begin{cases}
0,        & i = 1 ,\\[2pt]
u_{i-1}, & 2 \le i \le n+1 .
\end{cases}
.\] We can choose lowest weight (co)vectors of $V_{\varpi_k,\ZZ}$ for $1\leq k\leq n$ as \[u_{\varpi_k,-}=u_{n-k+2}\wedge\cdots\wedge u_{n+1}\in V_{\varpi_k,\ZZ}\] and \[u_{\varpi_k,+}^*=(u_1\wedge\cdots\wedge u_k)^*\in (V_{\varpi_k,\ZZ})^*,\] where the latter is given by extracting the coefficient of $u_1\wedge\cdots\wedge u_k$ under the standard basis of $\bigwedge^k_{\ZZ}\Std_{n+1,\ZZ}$. With this choice, one easily checks that $\langle e^{d_{\varpi_k}}\cdot u_{\varpi_k,-},u_{\varpi_k,+}^*\rangle >0$. 

For the computation of $S_{\varpi_k,j}$ and $T_{\varpi_k,j}$, we choose bases
of the one-dimensional vector spaces $\mathfrak{sl}_{n+1,e,j}$ and
$\mathfrak{sl}_{n+1,f,-j}$ via:
\begin{equation}
    e_j:=e^{(j)}\in \mathfrak{sl}_{n+1,e,j}
\end{equation}
\begin{equation}
    f_{j}\in \mathfrak{sl}_{n+1,f,-j},~ \kappa_{\Std_{n+1}}(f_{j},e_j)=1
\end{equation}
where $e^{(j)}$ means the $j$-th power of $e$ regarded as a $(n+1)\times(n+1)$-matrix.

In this way, define $s_{\lambda,j}:=S_{\lambda}(e_j)\in \QQ$ and $t_{\lambda,j}:=T_{\lambda}(f_j)\in \QQ$. We have 
\[s_{\varpi_k,j}=\langle e^{k(n+1-k)-j}e_j\cdot (u_{n-k+2}\wedge\cdots\wedge u_{n+1}),  (u_1\wedge\cdots\wedge u_k)^*   \rangle  \] \[t_{\varpi_k,j}=\langle \sum_{s=0}^{k(n+1-k)+j} e^sf_je^{k(n+1-k)+j-s}\cdot  (u_{n-k+2}\wedge\cdots\wedge u_{n+1}), (u_1\wedge\cdots\wedge u_k)^*\rangle\]
\[\epsilon_{\varpi_k,j}=\langle \sum_{s=0}^{k(n+1-k)} e^sf_je^{k(n+1-k)-s}e_j\cdot  (u_{n-k+2}\wedge\cdots\wedge u_{n+1}), (u_1\wedge\cdots\wedge u_k)^*\rangle.\]

\begin{prop}\label{prop:Aformulas}
    For $1\leq k\leq n$ and $1\leq j\leq n$, we have
    \begin{equation}
        d_{\varpi_k}=kk'
    \end{equation}
    \begin{equation}\label{eq:Adeg}
        \deg_{\varpi_k}=\frac{(kk')!}{\prod_{i=1}^{k}\prod_{j=1}^{k'} (i+j-1)}
    \end{equation}
    
    \begin{equation}\label{eq:Asformula}
    \begin{split}
        s_{\varpi_k,j}=\sum_{a=1}^{k}\sum_{\s\in S_k}\sgn(\sigma)  \binom{kk'-j}{k'+(\sigma(1)-1),\cdots,k'+(\sigma(a)-a)-j,\cdots,k'+(\sigma(k)-k)}
    \end{split}
    \end{equation}

    \begin{equation}\label{eq:Atformula}
    \begin{split}
        t_{\varpi_k,j}=\sum_{b=1}^{k}\sum_{\s\in S_k}\sgn(\sigma)  \binom{kk'+j+1}{k'+(\sigma(1)-1),\cdots,k'+(\sigma(b)-b)+j+1,\cdots,k'+(\sigma(k)-k)}
    \end{split}
    \end{equation}

    \begin{equation}\label{eq:Aepsilonformula}
    \begin{split}
        \epsilon_{\varpi_k,j}=\sum_{a=1}^k\sum_{b=1}^{k}\sum_{\s\in S_k}\sgn(\sigma)\cdot \\ \binom{kk'+1}{k'+(\sigma(1)-1), \cdots,k'+(\sigma(a)-a)-j,\cdots,k'+(\sigma(b)-b)+j+1,\cdots,k'+(\sigma(k)-k)}
    \end{split}
    \end{equation}
    where $k'=n+1-k$. Here, the multinomial coefficient is taken to be zero if any of the entries is negative. In the formula \eqref{eq:Aepsilonformula}, when $a=b$, the multinomial coefficient is understood as \[\binom{kk'+1}{k'+(\sigma(1)-1), \cdots,k'+(\sigma(a)-a)+1,\cdots,k'+(\sigma(k)-k)}.\]
\end{prop}

The proof of these formulas is elementary, and we omit it.
\begin{remark}
    The formula \eqref{eq:Aepsilonformula} should be compared with \cite[Theorem\,1.3.2]{feng2026eigenweightsarithmetichirzebruchproportionality}.
\end{remark}


\subsubsection{Type \texorpdfstring{$B_{n}$}{B_n}}\label{sec:Bepsilon}
In this subsection, we consider the case $G=\mathrm{SO}_{2n+1}$ in which case $d_j = 2j$ for $1\leq j\leq n$. We have $\Gc=\Sp_{2n}$. We order the fundamental weights as \[\dynkin[labels={\varpi_1,\varpi_2,,\varpi_{n-1},\varpi_n}, edge length=.5cm]{C}{}.\] Define $\widetilde{V}_{\varpi_k,\ZZ}=\bigwedge^{k}_{\ZZ}\Std_{2n,\ZZ}\in \Rep(\Sp_{2n,\ZZ})$ for $1\leq k\leq n$. Here $\widetilde{V}_{\varpi_1,\ZZ}=\Std_{2n,\ZZ}=\bigoplus_{i=1}^{2n}\ZZ\cdot u_i$ is the standard representation of $\Sp_{2n,\ZZ}$, and we choose the maximal torus of $\Sp_{2n,\ZZ}$ given by the prescribed basis of $\Std_{2n,\ZZ}$. The module $\Std_{2n,\ZZ}$ carries a symplectic form $\omega$ satisfying \[
\omega(u_i,u_j)=
\begin{cases}
1, & i+j=2n+1,\, 1\leq i\leq n,\\[2pt]
0, & i+j\neq 2n+1.
\end{cases}
\] Then $\Sp_{2n,\ZZ}=\Aut(\Std_{2n,\ZZ},\omega)$. The invariant bilinear form is given by $\kappa_{\min}=\frac{1}{2}\kappa_{\Std_{2n}}$.

Take $e\in \sp_{2n,\ZZ}$ such that \[e\cdot u_i=
\begin{cases}
    0,        & i = 1 ,\\[2pt]
u_{i-1}, & 2 \le i \le n ,\\[2pt]
2u_{i-1},   & i = n + 1 ,\\[2pt]
-u_{i-1},   & n + 2 \le i \leq 2n .
\end{cases}
\] We choose lowest weight (co)vectors \[u_{\varpi_k,-}=(-1)^{k(2n-k-1)/2}u_{2n-k+1}\wedge\cdots\wedge u_{2n}\in \widetilde{V}_{\varpi_k,\ZZ}\] and \[u_{\varpi_k,+}^*=(u_1\wedge\cdots\wedge u_k)^*\in (\widetilde{V}_{\varpi_k,\ZZ})^*\] for $1\leq k\leq n$. Take $V_{\varpi_k,\ZZ}\sub \widetilde{V}_{\varpi_k,\ZZ}$ to be the sub-representation of $\Sp_{2n,\ZZ}$ generated by $u_{\varpi_k,-}\in \widetilde{V}_{\varpi_k,\ZZ}$.  With this choice, one easily checks that $\langle e^{d_{\varpi_k}}\cdot u_{\varpi_k,-},u_{\varpi_k,+}^*\rangle >0$.

We choose bases of the one-dimensional vector spaces $\mathfrak{sp}_{2n,e,2j-1},\mathfrak{sp}_{2n,f,1-2j}$ via:
\begin{equation}
    e_j:=e^{(2j-1)}\in \mathfrak{sp}_{2n,e,2j-1}
\end{equation}
\begin{equation}
    f_{j}\in \mathfrak{sp}_{2n,f,1-2j},\, \kappa_{\Std_{2n}}(f_{j},e_j)=2
\end{equation}
as in type $A$ case. Define $s_{\lambda,j}:=S_{\lambda}(e_j)\in \QQ$ and $t_{\lambda,j}=T_{\lambda}(f_j)\in \QQ$.

\begin{prop}\label{prop:Bformulas}
    For $1\leq k\leq n$ and $1\leq j\leq n$, we have
    \begin{equation}
        d_{\varpi_k}=k(2n-k)
    \end{equation}
    \begin{equation}
        \deg_{\varpi_k}= 2^k\deg^{A_{2n-1}}_{\varpi_k}
    \end{equation}
    
    \begin{equation}
        s_{\varpi_k,j}=2^ks^{A_{2n-1}}_{\varpi_{k},2j-1}
    \end{equation}
    \begin{equation}
        t_{\varpi_k,j}=2^{k+1}t^{A_{2n-1}}_{\varpi_{k},2j-1}
    \end{equation}
    \begin{equation}
        \epsilon_{\varpi_k,j}=2^{k+1}\epsilon^{A_{2n-1}}_{\varpi_{k},2j-1}
    \end{equation}
    where we add superscript $A_{2n-1}$ to denote the corresponding invariant of type $A_{2n-1}$ given in Proposition \ref{prop:Aformulas}.
\end{prop}

This is an immediate consequence of the corresponding result in type $A$ given in Proposition \ref{prop:Aformulas}.

\subsubsection{Type \texorpdfstring{$C_{n}$}{C_n}}\label{sec:Cepsilon}
In this subsection, we consider the case $G=\mathrm{PSp}_{2n}$ in which case $d_j = 2j$ for $1\leq j\leq n$. We have $\Gc=\Spin_{2n+1}$. We order the fundamental weights as \[\dynkin[labels={\varpi_1,\varpi_2,,\varpi_{n-1},\varpi_n}, edge length=.5cm]{B}{}.\] Define $V_{\varpi_k,\ZZ}=\bigwedge^{k}_{\ZZ}\Std_{2n+1,\ZZ}$ for $1\leq k\leq n-1$ and $V_{2\varpi_n,\ZZ}=\bigwedge^n_{\ZZ}\Std_{2n+1,\ZZ}$. Here $V_{\varpi_1,\ZZ}=\Std_{2n+1,\ZZ}=\bigoplus_{i=1}^{2n+1}\ZZ\cdot u_i$ is the standard representation of $\SO_{2n+1,\ZZ}$. We choose the maximal torus of $\Spin_{2n+1,\ZZ}$ determined by the basis of $\Std_{2n+1,\ZZ}$. The module $\Std_{2n+1,\ZZ}$ carries a quadratic form $q$ with the corresponding symmetric bilinear form $(x,y)=q(x+y)-q(x)-q(y)$ satisfying \[
(u_i,u_j)=
\begin{cases}
1, & i+j=2n+2,\, i\neq n+1,\\[2pt]
2, & i=j=n+1, \\[2pt]
0, & i+j\neq 2n+2.
\end{cases}
\] Then $\SO_{2n+1,\ZZ}=\Aut(\Std_{2n+1,\ZZ},q)^{\circ}$ and $\Spin_{2n+1,\ZZ}$ is the universal covering group of $\SO_{2n+1,\ZZ}$. The invariant bilinear form is $\kappa_{\min}=\frac{1}{4}\kappa_{\Std_{2n+1}}$.

Take $e\in \so_{2n+1,\ZZ}$ such that \[e\cdot u_i=
\begin{cases}
    0,        & i = 1 ,\\[2pt]
2u_{i-1}, & 2 \le i \le n + 1 ,\\[2pt]
-u_{i-1},   & i = n + 2 ,\\[2pt]
-2u_{i-1},   & n + 3 \le i \leq 2n + 1 .

\end{cases}
\] We can choose lowest weight (co)vectors \[u_{\varpi_k,-}=(-1)^{k(2n-k+1)/2}u_{2n-k+2}\wedge\cdots\wedge u_{2n+1}\in V_{\varpi_k,\ZZ}\]\[u_{\varpi_k,+}^*=(u_1\wedge\cdots\wedge u_k)^*\in (V_{\varpi_k,\ZZ})^*\] for $1\leq k\leq n-1$, and \[u_{2\varpi_n,-}=(-1)^{n(n+1)/2}u_{n+2}\wedge\cdots\wedge u_{2n+1}\in V_{2\varpi_n,\ZZ}\] \[u_{2\varpi_n,+}^*=(u_1\wedge\cdots\wedge u_n)^*\in (V_{2\varpi_n,\ZZ})^*.\] One can take $V_{\varpi_n,\ZZ}\in \Rep(\Spin_{2n+1,\ZZ})$ with lowest weight (co)vectors \[u_{\varpi_n,-}\in V_{\varpi_n,\ZZ},\,u_{\varpi_n,+}^*\in (V_{\varpi_n,\ZZ})^*\] equipped with an injection $V_{2\varpi_n,\ZZ}\sub  V_{\varpi_n,\ZZ}^{\otimes 2}$ such that \[u_{\varpi_n,-}^{\otimes 2}= u_{2\varpi_n,-}, \,(u_{\varpi_n,+}^*)^{\otimes 2}= u_{2\varpi_n,+}^*.\] One easily checks that $\langle e^{d_{\varpi_k}}\cdot u_{\varpi_k,-},u_{\varpi_k,+}^*\rangle >0$ for $1\leq k\leq n$.

We choose bases of the one-dimensional vector spaces $\mathfrak{so}_{2n+1,e,2j-1},\mathfrak{so}_{2n+1,f,1-2j}$ via:
\begin{equation}
    e_j:=e^{(2j-1)}\in \mathfrak{so}_{2n+1,e,2j-1}
\end{equation}
\begin{equation}
    f_{j}\in \mathfrak{so}_{2n+1,f,1-2j},\, \kappa_{\Std_{2n+1}}(f_{j},e_j)=4
.\end{equation}
 Define $s_{\lambda,j}:=S_{\lambda}(e_j)\in \QQ$ and $t_{\lambda,j}=T_{\lambda}(f_j)\in \QQ$.

\begin{prop}\label{prop:Cformulas}
    For $1\leq k\leq n-1$ and $1\leq j\leq n$, we have
    \begin{equation}
        d_{\varpi_k}=k(2n+1-k)
    \end{equation}
    \begin{equation}
        \deg_{\varpi_k}= 2^{k(2n-k)}\deg^{A_{2n}}_{\varpi_k}
    \end{equation}
    
    \begin{equation}
        s_{\varpi_k,j}=2^{k(2n-k)}s^{A_{2n}}_{\varpi_{k},2j-1}
    \end{equation}
    \begin{equation}
        t_{\varpi_k,j}=2^{k(2n-k)+2}t^{A_{2n}}_{\varpi_{k},2j-1}
    \end{equation}
    \begin{equation}
        \epsilon_{\varpi_k,j}=2^{k(2n-k)+2}\epsilon^{A_{2n}}_{\varpi_{k},2j-1}.
    \end{equation}

    For $k=n$ and $1\leq j\leq n$, we have 
        \begin{equation}
        d_{\varpi_n}=n(n+1)/2
    \end{equation}
    \begin{equation}\label{eq:Cspindeg}
        \deg_{\varpi_n}= (\binom{2d_{\varpi_n}}{d_{\varpi_n}}^{-1}2^{n^2}\deg_{\varpi_n}^{A_{2n}})^{1/2}
    \end{equation}
    
    \begin{equation}
        s_{\varpi_n,j}=\binom{2d_{\varpi_n}-2j+1}{d_{\varpi_n}}^{-1} \deg_{\varpi_n}^{-1}  2^{n^2-1}s^{A_{2n}}_{\varpi_{n},2j-1}
    \end{equation}
    \begin{equation}
        t_{\varpi_n,j}=\binom{2d_{\varpi_n}+2j}{d_{\varpi_n}}^{-1} \deg_{\varpi_n}^{-1}  2^{n^2+1}t^{A_{2n}}_{\varpi_{n},2j-1}
    \end{equation}
    \begin{equation}\label{eq:Cspinepsilon}
        \epsilon_{\varpi_n,j}=\binom{2d_{\varpi_n}+1}{d_{\varpi_n}}^{-1}\deg_{\varpi_n}^{-1} (2^{n^2+1}\epsilon^{A_{2n}}_{\varpi_{n},2j-1}-\binom{2d_{\varpi_n}+1}{d_{\varpi_n}-2j+1}s_{\varpi_n,j}t_{\varpi_n,j})
    \end{equation}
\end{prop}
 This follows immediately from Proposition \ref{prop:Aformulas} and Proposition \ref{prop:redtofundepsilon}.

\begin{remark}
    The formula \eqref{eq:Cspinepsilon} should be compared with \cite[Theorem\,1.3.6]{feng2026eigenweightsarithmetichirzebruchproportionality}.
\end{remark}

 \subsubsection{Type \texorpdfstring{$D_{n}$}{D_n}}\label{sec:Depsilon}
In this subsection, we consider the case $G=\mathrm{PSO}_{2n}$ in which case we arrange $d_j = 2j$ for $1\leq j\leq n-1$ and $d_n=n$. Note that when $n$ is even, we have $d_{n/2}=d_n=n$. We have $\Gc=\Spin_{2n}$ whose fundamental weights are ordered as \[\dynkin[labels={\varpi_1,\varpi_2,,,\varpi_{n-1},\varpi_n}, edge length=.5cm]{D}{}.\] Define $V_{\varpi_k,\ZZ}=\bigwedge^{k}_{\ZZ}\Std_{2n,\ZZ}\in \Rep(\Spin_{2n,\ZZ})$ for $1\leq k\leq n-2$ and $V_{\varpi_{n-1}+\varpi_n,\ZZ}= \bigwedge^{n-1}_{\ZZ}\Std_{2n,\ZZ}\in\Rep(\Spin_{2n,\ZZ})$. Here $V_{\varpi_1,\ZZ}=\Std_{2n,\ZZ}=\bigoplus_{i=1}^{2n}\mathbb{Z}\cdot u_i$ is the standard representation of $\SO_{2n}$. We choose the maximal torus of $\Spin_{2n,\ZZ}$ determined by the basis of $\Std_{2n,\ZZ}$. The module $\Std_{2n,\ZZ}$ carries a quadratic form $q$ with the corresponding symmetric bilinear form $(x,y)=q(x+y)-q(x)-q(y)$ satisfying \[
(u_i,u_j)=
\begin{cases}
1, & i+j=2n+1,\\[2pt]
0, & i+j\neq 2n+1.
\end{cases}
\] Then $\SO_{2n,\ZZ}=\Aut(\Std_{2n,\ZZ},q)^{\circ}$ and $\Spin_{2n,\ZZ}$ is the universal covering group of $\SO_{2n,\ZZ}$. The invariant bilinear form on $\so_{2n}$ is given by $\kappa_{\min}=\frac{1}{2}\kappa_{\Std_{2n}}$.

Take $e\in \so_{2n,\ZZ}$ such that \[e\cdot u_i=
\begin{cases}
0,        & i = 1 ,\\[2pt]
u_{i-1}, & 2 \le i \le n ,\\[2pt]
u_{i-2},   & i=n+1,\\[2pt]
-u_{i-2}-u_{i-1},   & i = n+2 ,\\[2pt]
-u_{i-1},   & n+3\leq i\leq 2n.
\end{cases}
\] We can choose lowest weight (co)vectors \[u_{\varpi_k,-}=(-1)^{k(2n-k-1)/2}u_{2n-k+1}\wedge\cdots\wedge u_{2n}\in V_{\varpi_k,\ZZ}\]\[u_{\varpi_k,+}^*=(u_1\wedge\cdots\wedge u_k)^*\in (V_{\varpi_k,\ZZ})^*\] for $1\leq k\leq n-2$, and \[u_{\varpi_{n-1}+\varpi_n,-}=(-1)^{n(n-1)/2}u_{n+2}\wedge\cdots\wedge u_{2n}\in V_{\varpi_{n-1}+\varpi_n,\ZZ}\] \[u_{\varpi_{n-1}+\varpi_n,+}^*=(u_1\wedge\cdots\wedge u_{n-1})^*\in (V_{\varpi_{n-1}+\varpi_n,\ZZ})^*.\] 
Consider elements \[u_{2\varpi_{n-1},-}'=(-1)^{n(n-1)/2}u_{n+1}\wedge\cdots\wedge u_{2n}\in \bigwedge^n_{\ZZ}\Std_{2n,\ZZ}\] 
\[u_{2\varpi_{n-1},+}^*=(u_1\wedge\cdots\wedge u_{n})^*\in (\bigwedge^n_{\ZZ}\Std_{2n,\ZZ})^*\] 
\[u_{2\varpi_{n},-}'=(-1)^{n(n-1)/2}u_{n}\wedge u_{n+2}\wedge\cdots\wedge u_{2n}\in \bigwedge^n_{\ZZ}\Std_{2n,\ZZ}\] 
\[u_{2\varpi_{n},+}^*=(u_1\wedge\cdots\wedge u_{n-1}\wedge u_{n+1})^*\in (\bigwedge^n_{\ZZ}\Std_{2n,\ZZ})^*.\] 
Take $u_{2\varpi_{n-1},-} = u_{2\varpi_{n-1},-}'$ when $n$ is even and $u_{2\varpi_{n-1},-} = u_{2\varpi_{n},-}'$ when $n$ is odd. Similarly, take $u_{2\varpi_{n},-} = u_{2\varpi_{n},-}'$ when $n$ is even and $u_{2\varpi_{n},-} = u_{2\varpi_{n-1},-}'$ when $n$ is odd.
Let $V_{2\varpi_{n-1},\ZZ},\,V_{2\varpi_n,\ZZ}\sub \bigwedge^n_{\ZZ}\Std_{2n,\ZZ}$ be the subrepresentations of $\Spin_{2n,\ZZ}$ generated by $u_{2\varpi_{n-1},-},\,u_{2\varpi_{n},-} \in \bigwedge^n_{\ZZ}\Std_{2n,\ZZ}$, respectively. With these choices, one can take $V_{\varpi_{n-1},\ZZ},\,V_{\varpi_{n},\ZZ}\in \Rep(\Spin_{2n,\ZZ})$ with lowest weight (co)vectors \[u_{\varpi_{n-1},-}\in V_{\varpi_{n-1},\ZZ},\, u_{\varpi_{n-1},+}^*\in (V_{\varpi_{n-1},\ZZ})^*\] \[u_{\varpi_{n},-}\in V_{\varpi_{n},\ZZ},\, u_{\varpi_{n},+}^*\in (V_{\varpi_{n},\ZZ})^*\] such that there exist embeddings $V_{2\varpi_{n-1},\ZZ}\sub V_{\varpi_{n-1},\ZZ}^{\otimes 2},\, V_{2\varpi_{n},\ZZ}\sub V_{\varpi_{n},\ZZ}^{\otimes 2}$ such that 
\[u_{\varpi_{n-1},-}^{\otimes 2}=u_{2\varpi_{n-1},-},\, (u_{\varpi_{n-1},+}^*)^{\otimes 2}=u_{2\varpi_{n-1},+}^*,\, u_{\varpi_{n},-}^{\otimes 2}=u_{2\varpi_{n},-},\, (u_{\varpi_{n},+}^*)^{\otimes 2}=u_{2\varpi_{n},+}^*.\] 
Moreover, there exists an embedding $V_{\varpi_{n-1}+\varpi_n,\ZZ}\sub V_{\varpi_{n-1},\ZZ}\otimes V_{\varpi_n,\ZZ}$ such that \[u_{\varpi_{n-1},-}\otimes u_{\varpi_{n},-}=u_{\varpi_{n-1}+\varpi_n,-},\, u_{\varpi_{n-1},+}^*\otimes u_{\varpi_{n},+}^*=u_{\varpi_{n-1}+\varpi_n,+}^*.\] With these choices, one easily checks that $\langle e^{d_{\varpi_k}}\cdot u_{\varpi_k,-},u_{\varpi_k,+}^*\rangle >0$ for $1\leq k\leq n$.

For $1\leq j\leq n-1$, define $e_j\in \mathfrak{so}_{2n,e,2j-1},f_j\in \mathfrak{so}_{2n,f,1-2j}$ via:
\begin{equation}
    e_j:=e^{(2j-1)}\in \mathfrak{so}_{2n,e,2j-1}
\end{equation}
\begin{equation}
    f_{j}\in \Span(f^{(2j-1)})\sub \mathfrak{so}_{2n,f,1-2j},\, \kappa_{\Std_{2n}}(f_{j},e_j)=2
.\end{equation}
For $j=n$, take $e_n\in \so_{2n,e,n-1}$ such that \[e_n\cdot u_i=
\begin{cases}
u_{1},   & i = n ,\\[2pt]
-u_{1},   & i = n + 1,\\[2pt]
u_{n}-u_{n+1}, & i=2n ,\\[2pt]
0,   & \textup{else}\\[2pt]
\end{cases}
\] and $f_n\in \so_{2n,f,1-n}$ such that \[f_n\cdot u_i=
\begin{cases}
(u_{n} - u_{n+1})/2,   & i = 1,\\[2pt]
u_{2n}/2, & i=n ,\\[2pt]
-u_{2n}/2,   & i = n + 1 ,\\[2pt]
0,   & \textup{else}.\\[2pt]
\end{cases}
\] Then one has $\kappa_{\Std_{2n}}(f_{n},e_n)=2$. Moreover, when $n$ is even, one has $\kappa_{\Std_{2n}}(f_{n},e_{n/2})=\kappa_{\Std_{2n}}(f_{n/2},e_n)=0$. Define $s_{\lambda,j}:=S_{\lambda}(e_j)\in\QQ$ and $t_{\lambda,j}=T_{\lambda}(f_j)\in\QQ$.

Take $\sigma\in \Aut(\Spin_{2n,\ZZ})$ to be the unique non-trivial outer automorphism fixing the pinning determined by $e\in \so_{2n,\ZZ}$ and the first vertex of the Dynkin diagram. We have $\sigma(e_j)=e_j$, $\sigma(f_j)=f_j$ for $1\leq j\leq n-1$, $\sigma(e_n)=-e_n$, and $\sigma(f_n)=-f_n$.

When $n$ is odd, we have $\epsilon_{\varpi_k,j}=\kappa_{\varpi_k}(f_j,e_j)$ for all $j,k$. When $n$ is even, the same formula holds for $j\neq n/2,n$ or $k\neq n-1,n$ while the numbers $\epsilon_{\lambda,n/2}$ and $\epsilon_{\lambda,n}$ for $\lambda=\varpi_{n-1}, \varpi_n$ are eigenvalues of the $2\times 2$-matrix \begin{equation}\label{eq:D:22mat}\begin{pmatrix}
\kappa_{\lambda}(f_{n/2},e_{n/2}) & \kappa_{\lambda}(f_{n},e_{n/2}) \\
\kappa_{\lambda}(f_{n/2},e_n) & \kappa_{\lambda}(f_n, e_n)
\end{pmatrix}.\end{equation}

\begin{prop}\label{prop:Dformulas}
    For $1\leq k\leq n-2$ and $1\leq j\leq n-1$, we have
    \begin{equation}
        d_{\varpi_k}=k(2n-1-k),
    \end{equation}
    \begin{equation}
        \deg_{\varpi_k}= 2^{k}\deg^{A_{2n-2}}_{\varpi_k},
    \end{equation}
    
    \begin{equation}
        s_{\varpi_k,j}=2^{k}s^{A_{2n-2}}_{\varpi_{k},2j-1},
    \end{equation}
    \begin{equation}
        t_{\varpi_k,j}=2^{k+1}t^{A_{2n-2}}_{\varpi_{k},2j-1},
    \end{equation}
    \begin{equation}
        \epsilon_{\varpi_k,j}=2^{k + 1}\epsilon^{A_{2n-2}}_{\varpi_{k},2j-1}.
    \end{equation}
    For $1\leq k \leq n-2$ and $j=n$, we have     \begin{equation}
        s_{\varpi_k,n}=0,
    \end{equation}
    \begin{equation}
        t_{\varpi_k,n}=0,
    \end{equation}
    \begin{equation}
        \epsilon_{\varpi_k,n}=2^{k}\sum_{\s\in S_k}\sgn(\s)\binom{k(2n-1-k)+1}{2n-k+(\s(1)-1), 2n-k-1+(\s(2)-2),\cdots,2n-k-1+(\s(k)-k)}.
    \end{equation}  
    For $k=n-1, n$ and $1\leq j\leq n-1$, we have 
        \begin{equation}
        d_{\varpi_k}=n(n-1)/2,
    \end{equation}
    \begin{equation}\label{eq:Dspindeg}
        \deg_{\varpi_k} = (\binom{2d_{\varpi_{n-1}}}{d_{\varpi_{n-1}}}^{-1}2^{n-1}\deg_{\varpi_{n-1}}^{A_{2n-2}})^{1/2},
    \end{equation}
    
    \begin{equation}
        s_{\varpi_k,j}=\binom{2d_{\varpi_{n-1}}-2j+1}{d_{\varpi_{n-1}}}^{-1} \deg_{\varpi_{n-1}}^{-1}  2^{n-2}s^{A_{2n-2}}_{\varpi_{n-1},2j-1},
    \end{equation}
    \begin{equation}
        t_{\varpi_k,j}=\binom{2d_{\varpi_{n-1}}+2j}{d_{\varpi_{n-1}}}^{-1} \deg_{\varpi_{n-1}}^{-1}  2^{n-1}t^{A_{2n-2}}_{\varpi_{n-1},2j-1},
    \end{equation}
    \begin{equation}
        \kappa_{\varpi_k}(f_j,e_j)=\binom{2d_{\varpi_{n-1}}+1}{d_{\varpi_{n-1}}}^{-1}\deg_{\varpi_{n-1}}^{-1} (2^{n-1}\epsilon^{A_{2n-2}}_{\varpi_{n-1},2j-1}-\binom{2d_{\varpi_{n-1}}+1}{d_{\varpi_{n-1}}-2j+1}s_{\varpi_{n-1},j}t_{\varpi_{n-1},j})
    .\end{equation}
    For $k=n-1, n$ and $j=n$, we have    
    \begin{equation}
        s_{\varpi_{n-1},n}=-s_{\varpi_{n},n}=(-1)^{n-1}\binom{2d_{\varpi_{n-1}}-n+1}{d_{\varpi_{n-1}}}^{-1} \deg_{\varpi_{n-1}}^{-1}  2^{n-2}\deg^{A_{2n-3}}_{\varpi_{n-1}}
    ,\end{equation}
    \begin{equation}
    \begin{split}
        t_{\varpi_{n-1},n}=-t_{\varpi_{n},n}= \binom{2d_{\varpi_{n-1}}+n}{d_{\varpi_{n-1}}}^{-1} \deg_{\varpi_{n-1}}^{-1}  2^{n-2}  \sum_{\s\in S_n}\sgn(\s)\cdot \\ \binom{n^2}{n-1+(\s(1)-1),\cdots,n-1+(\s(n-1)-(n-1)),n-1+\s(n)}
        \end{split}
    ,\end{equation}
    \begin{equation}
        \kappa_{\varpi_k}(f_j,e_j)=\binom{2d_{\varpi_{n-1}}+1}{d_{\varpi_{n-1}}}^{-1}\deg_{\varpi_{n-1}}^{-1} (2^{-1}\epsilon_{\varpi_{n-1}+\varpi_n,n}+\binom{2d_{\varpi_{n-1}}+1}{d_{\varpi_{n-1}}-n+1}s_{\varpi_{n-1},j}t_{\varpi_{n-1},j})
    \end{equation}
    where
    \begin{equation}
        \epsilon_{\varpi_{n-1}+\varpi_n,n}=2^{n-1}\sum_{\s\in S_{n-1}}\sgn(\s)\binom{n(n-1)+1}{n+1+(\s(1)-1), n+(\s(2)-2),\cdots,n+(\s(n-1)-(n-1))}.
    \end{equation}  

    Finally, when $n$ is even, $k=n-1, n$, we have
    \begin{equation}
        \begin{split}
        \kappa_{\varpi_{n-1}}(f_n,e_{n/2})=-\kappa_{\varpi_{n}}(f_n,e_{n/2})= \\ \binom{2d_{\varpi_{n-1}}+1}{d_{\varpi_{n-1}}}^{-1}\deg_{\varpi_{n-1}}^{-1}(2^{-1}\kappa_{2\varpi_{n-1}}(f_n,e_{n/2})-\binom{2d_{\varpi_{n-1}}+1}{d_{\varpi_{n-1}}+1-n } t_{\varpi_{n-1}, n}s_{\varpi_{n-1}, n/2})
        \end{split}
    \end{equation}
    \begin{equation}
        \begin{split}
        \kappa_{\varpi_{n-1}}(f_{n/2},e_{n})=-\kappa_{\varpi_{n}}(f_{n/2},e_{n})= \\ \binom{2d_{\varpi_{n-1}}+1}{d_{\varpi_{n-1}}}^{-1}\deg_{\varpi_{n-1}}^{-1}(2^{-1}\kappa_{2\varpi_{n-1}}(f_{n/2},e_{n})-\binom{2d_{\varpi_{n-1}}+1}{d_{\varpi_{n-1}}+1-n } t_{\varpi_{n-1}, n/2}s_{\varpi_{n-1}, n})
        \end{split}
    \end{equation}
    where
    \begin{equation}
        \begin{split}
        \kappa_{2\varpi_{n-1}}(f_n,e_{n/2})=2^{n-1}\sum_{a=1}^{n-1}\sum_{\s\in S_n}\sgn(\s)\cdot \\ \binom{n(n-1)+1}{n-1+(\s(1)-1),\cdots,\s(a)-a,\cdots,n-1+(\s(n-1)-(n-1)),n-1+\s(n)}
        \end{split}
    \end{equation}
    \begin{equation}
        \kappa_{2\varpi_{n-1}}(f_{n/2},e_n)=-2^n t^{A_{2n-3}}_{\varpi_{n-1},n-1}.
    \end{equation}

\end{prop}
When $j\neq n$, these formulas follow immediately from Proposition \ref{prop:Aformulas} and Proposition \ref{prop:redtofundepsilon}. The case $j=n$ can be treated similarly.

\begin{remark}
    The formulas in Proposition \ref{prop:Dformulas} can be compared with the formulas in \cite[Theorem\,1.3.8]{feng2026eigenweightsarithmetichirzebruchproportionality}.
\end{remark}

\subsubsection{Small rank cases}
In this subsection, we list the invariants $\epsilon_{\varpi_i,j}$ for all almost-simple simply-connected groups $\Gc$ of classical types with rank not exceeding 4. These invariants together with the relevant invariants $d_{\lambda},\deg_{\lambda}$ are given in Table \ref{tab:rk1epsilon}, Table \ref{tab:rk2epsilon}, Table \ref{tab:rk3epsilon}, Table \ref{tab:rk4epsilon}.


\begin{table}[ht]
    \centering
    \caption{Value of $\deg_{\lambda}, d_{\lambda}, \epsilon_{\lambda,1}$ for classical groups of rank $1$}
    \label{tab:rk1epsilon}
     \[ \begin{array}{c|c|c|c|c}
     \Gc &  \textup{$\lambda$} & \deg_{\lambda} &d_{\lambda} & \epsilon_{\lambda,1} \\
     \hline
     A_1 &  \varpi_1 & 1 & 1 & 1 \\
     \end{array}\]
\end{table}

\begin{table}[ht]
    \centering
    \caption{Value of $\deg_{\lambda}, d_{\lambda}, (\epsilon_{\lambda,j})_{j=1}^2$ for classical groups of rank $2$}
    \label{tab:rk2epsilon}
     \[ \begin{array}{c|c|c|c|c}
     \Gc &  \textup{$\lambda$} & \deg_{\lambda} &d_{\lambda} & (\epsilon_{\lambda,j})_{j=1}^2 \\
     \hline
     A_2 &  \varpi_1 & 1 & 2 & (1,1) \\
     A_2 &  \varpi_2 & 1 & 2 & (1,1) \\
     \hline
     B_2 &  \varpi_1 & 8 & 4 & (32,32) \\
     B_2 &  \varpi_2 & 2 & 3 & (4,4) \\
     \end{array}\]
\end{table}

\begin{table}[ht]
    \centering
    \caption{Value of $\deg_{\lambda}, d_{\lambda}, (\epsilon_{\lambda,j})_{j=1}^3$ for classical groups of rank $3$}
    \label{tab:rk3epsilon}
     \[ \begin{array}{c|c|c|c|c}
     \Gc &  \textup{$\lambda$} & \deg_{\lambda} &d_{\lambda} & (\epsilon_{\lambda,j})_{j=1}^3 \\
     \hline
     A_3 &  \varpi_1 & 1 & 3 & (1,1,1) \\
     A_3 &  \varpi_2 & 2 & 4 & (4,2,4) \\
     A_3 &  \varpi_3 & 1 & 3 & (1,1,1) \\
     \hline
     B_3 &  \varpi_1 & 32 & 6 & (128,128,128) \\
     B_3 &  \varpi_2 & 10752 & 10 & (168960,67584,125952) \\
     B_3 &  \varpi_3 & 16 & 6 & (64,40,64) \\
     \hline
     C_3 &  \varpi_1 & 2 & 5 & (4,4,4) \\
     C_3 &  \varpi_2 & 56 & 8 & (384,184,328) \\
     C_3 &  \varpi_3 & 336 & 9 & (3168,1568,2720) \\
     \end{array}\]
\end{table}

\begin{table}[ht]
    \centering
    \caption{Value of $\deg_{\lambda}, d_{\lambda}, (\epsilon_{\lambda,j})_{j=1}^4$ for classical groups of rank $4$}
    \label{tab:rk4epsilon}
     \[ \begin{array}{c|c|c|c|c}
     \Gc &  \textup{$\lambda$} & \deg_{\lambda} &d_{\lambda} & (\epsilon_{\lambda,j})_{j=1}^4 \\
     \hline
     A_4 &  \varpi_1 & 1 & 4 & (1,1,1,1) \\
     A_4 &  \varpi_2 & 5 & 6 & (14,6,9,13) \\
     A_4 &  \varpi_3 & 5 & 6 & (14,6,9,13) \\
     A_4 &  \varpi_4 & 1 & 4 & (1,1,1,1) \\
     \hline
     B_4 &  \varpi_1 & 128 & 8 & (512,512,512,512) \\
     B_4 &  \varpi_2 & 1757184 & 14 & (32800768,11501568,19169280,23216128) \\
     B_4 &  \varpi_3 & 2867724288 & 18 & (108973522944,26417823744,42420928512,65536524288) \\
     B_4 &  \varpi_4 & 768 & 10 & (5632,2432 ,3584 ,5248) \\
     \hline
     C_4 &  \varpi_1 & 2  & 7 & (4, 4,4 , 4) \\
     C_4 &  \varpi_2 & 528 & 12 & (4576,1664,3008,3424) \\
     C_4 &  \varpi_3 & 48048 & 15 & (777920,232128,349632,555968) \\
     C_4 &  \varpi_4 & 384384 & 16 & (7468032 ,2376192,3355392,5392128) \\
     \hline
     D_4 &  \varpi_1 & 2 & 6 & (4,4,4,2) \\
     D_4 &  \varpi_2 & 168 & 10 & (1320,528,984,528) \\
     D_4 &  \varpi_3 & 2 & 6 & (4,4,4,2) \\
     D_4 &  \varpi_4 & 2 & 6 & (4,4,4,2) \\
     \end{array}\]
\end{table}

\subsubsection{Exceptional types}
In this subsection, we consider the case where $G$ is of exceptional type. Using SageMath \cite{sagemath}, one can calculate the numbers $(\epsilon_{\lambda,j})_{j=1}^n$ when $\lambda\in X^*(\Tc)_+$ is the (quasi-)minuscule or adjoint weight. Table \ref{tab:Edeg} gives the numbers $d_{\lambda}, \deg_{\lambda}$ and Table \ref{tab:Eepsilon} gives the numbers $(\epsilon_{\lambda,j})_{j=1}^n$.

\begin{remark}
    Our computation suggests that $\epsilon_{\lambda,i}$ are always algebraic integers. Moreover, when $G$ is not of type $D_n$ for $n\geq 4$ even, they are always positive integers. We do not have a conceptual explanation for this phenomenon.
\end{remark}

\begin{table}[ht]
    \centering
    \caption{Value of $\deg_{\lambda}, d_{\lambda}$ for exceptional types}
    \label{tab:Edeg}

     \[ \begin{array}{c|c|c|c}
     \Gc &  \textup{Name of $\lambda$} & \deg_{\lambda} &d_{\lambda} \\
     \hline
     G_2 &  \textup{quasi-minuscule} & 18 & 6 \\
     G_2 &  \textup{adjoint} & 13608 & 10 \\
     F_4 &  \textup{quasi-minuscule} & 4992 & 16 \\
     F_4 &  \textup{adjoint} & 154791936 & 22 \\
     E_6 &  \textup{minuscule} & 78 & 16 \\
     E_6 &  \textup{adjoint} & 151164 & 22 \\
     E_7 &  \textup{minuscule} & 13110 & 27 \\
     E_7 &  \textup{adjoint} & 141430680 & 34 \\
     E_8 &  \textup{adjoint} & 126937516885200 & 58

     \end{array}\]
     \end{table}

\begin{table}[ht]
\centering
\caption{Values of $\epsilon_{\lambda,j}$ for exceptional types}
\label{tab:Eepsilon}
\[\begin{array}{c|c|c}
    \Gc & \textup{Name of $\lambda$} & (\epsilon_{\lambda,j})_{j=1}^{n} \\
    \hline
     G_2 & \textup{quasi-minuscule} & (108, 108) \\
     G_2 & \textup{adjoint} & (320760, 239112)  \\
     F_4 & \textup{quasi-minuscule} & (52224, 27648, 32640, 47232) \\
     F_4 & \textup{adjoint} & (4016652288, 1610956800, 2099613696, 2908827648)  \\
     E_6 & \textup{minuscule} & (408, 221, 216, 255, 299, 369)\\
     E_6 & \textup{adjoint} & (1961256, 490314, 786600, 1025202, 967518, 1420326)  \\
     E_7 & \textup{minuscule} & (120060, 58320, 48024, 69920, 67176, 85376, 101744) \\
     
     E_7 & \textup{adjoint} & \begin{aligned}(2530864800,
 643719960,
 1025744616,
 921331160, \\
 1301226984,
 1386811784,
 1807586936)\end{aligned}  \\
     E_8 & \textup{adjoint} & \begin{aligned}(3503065990170600,
 1035140220518880,
 1116031452545520,
 1289228099378520, \\
 1602802318771080,
 1661879186158800,
 2045534982573600,
 2471069708566200)\end{aligned}
\end{array}.\]
\end{table}

\subsection{Examples of \texorpdfstring{$b_{\lambda}$}{b}}\label{sec:bexamples}
In this section, we give examples of the number $b_{\lambda}\in \QQ$ for $\lambda\in X^*(\Tc)_+$.

\subsubsection{Reduction to fundamental weights}
The calculation of the numbers $b_{\lambda}$ can be reduced to the calculation of a sum of two fundamental weights $\varpi_i+\varpi_j \in X^*(\Tc)_+$. In fact, the calculation for $\lambda_1 +\lambda_2 + \lambda_3 \in X^*(\Tc)_+$ can be reduced to the calculation of $\lambda_k, \lambda_i+\lambda_j\in X^*(\Tc)_+$ for $1\leq i,j,k\leq 3$. This reduction procedure is provided by Proposition \ref{prop:redtofundb}.

\begin{prop}\label{prop:redtofundb}
For $\lambda_1,\lambda_2,\lambda_3\in X^*(\Tc)_+$, we have
\begin{equation}
\begin{split}
    b_{\lambda_1+\lambda_2+\lambda_3} = \\ \binom{d_{\lambda_1}+d_{\lambda_2}+d_{\lambda_3}+1}{d_{\lambda_1}}\deg_{\lambda_1}b_{\lambda_2+\lambda_3}+\binom{d_{\lambda_1}+d_{\lambda_2}+d_{\lambda_3}+1}{d_{\lambda_2}}\deg_{\lambda_2}b_{\lambda_1+\lambda_3}\\+\binom{d_{\lambda_1}+d_{\lambda_2}+d_{\lambda_3}+1}{d_{\lambda_3}}\deg_{\lambda_3}b_{\lambda_1+\lambda_2} -\binom{d_{\lambda_1}+d_{\lambda_2}+d_{\lambda_3}+1}{d_{\lambda_1}+1}b_{\lambda_1}\deg_{\lambda_2+\lambda_3}\\-\binom{d_{\lambda_1}+d_{\lambda_2}+d_{\lambda_3}+1}{d_{\lambda_2}+1}b_{\lambda_2}\deg_{\lambda_1+\lambda_3}-\binom{d_{\lambda_1}+d_{\lambda_2}+d_{\lambda_3}+1}{d_{\lambda_3}+1}b_{\lambda_3}\deg_{\lambda_1+\lambda_2}
    \end{split}.
\end{equation}   
\end{prop} This follows immediately from Theorem \ref{thm:bspectral}.

\subsubsection{Minuscule case}\label{sec:bmin}
When $\lambda\in X^*(\Tc)_+$ is minuscule, the number $b_{\lambda}$ is particularly easy to determine:
\begin{prop}\label{prop:minb}
    When $\lambda\in X^*(\Tc)_+$ is minuscule, we have \[b_{\lambda} = \frac{1}{2}\kappa_{\min}(\lambda,\lambda)(d_{\lambda} +1)\deg_{\lambda}.\]
\end{prop} This follows immediately from Theorem \ref{thm:bspectral}. The following is a table of the invariants involved in the formula of $b_{\lambda}$ for all minuscule weights $\lambda\in X^*(\Tc)$:
   \[
\begin{array}{c|c|c|c|c}
\Gc & \textup{Name of $\lambda$} & \kappa_{\min}(\lambda,\lambda) & d_{\lambda} & \deg_{\lambda}   \\
\hline
A_{n}(n\geq 1) & \varpi_k (1\leq k\leq n) & k(n+1-k)/(n+1) & k(n+1-k) & \text{\eqref{eq:Adeg}}   \\
B_{n}(n\geq 1) & \textup{spin} & n/2 & n(n+1)/2  & \text{\eqref{eq:Cspindeg}}   \\
C_{n}(n\geq 2) & \textup{standard} & 1 & 2n-1 & 2 \\
D_n (n\geq 3) & \textup{standard} & 1 & 2n-2 & 2 \\
D_n (n\geq 3) & \textup{half-spin} & n/4 & n(n-1)/2 & \text{\eqref{eq:Dspindeg}} \\
E_6 & \textup{minuscule} & 4/3 &  16 & 78 \\
E_7 & \textup{minuscule} & 3/2  & 27 & 13110
\end{array}
.\] 

\subsubsection{Classical types}
When $\Gc$ is of classical type, one can write down formulas parallel to those in \S\ref{sec:Aepsilon}, \S\ref{sec:Bepsilon}, \S\ref{sec:Cepsilon}, and \S\ref{sec:Depsilon}. We omit the details.

\subsubsection{Exceptional types}
When $\Gc$ is of exceptional type, with the help of SageMath \cite{sagemath}, we calculate the number $b_{\lambda}$ when $\lambda\in X^*(\Tc)_+$ is the (quasi-)minuscule or adjoint weight. The result is collected in Table \ref{tab:Eb}.

\begin{table}[ht]
\centering
\caption{Values of $b_{\lambda}$ for exceptional types}
\label{tab:Eb}
\[\begin{array}{c|c|c}
    \Gc & \textup{Name of $\lambda$} & b_{\lambda} \\
    \hline
     G_2 & \textup{quasi-minuscule} & 108 \\
     G_2 & \textup{adjoint} & 279936  \\
     F_4 & \textup{quasi-minuscule} & 79872 \\
     F_4 & \textup{adjoint} & 5318025216  \\
     E_6 & \textup{minuscule} & 884\\
     E_6 & \textup{adjoint} & 3325608  \\
     E_7 & \textup{minuscule} & 275310 \\
     
     E_7 & \textup{adjoint} & 4808643120  \\
     E_8 & \textup{adjoint} & 7362375979341600
\end{array}.\]
\end{table}

\bibliographystyle{amsalpha}
\bibliography{Bibliography}

@misc{BZSV,
      title={Relative {Langlands} duality}, 
      author={David Ben-Zvi and Yiannis Sakellaridis and Akshay Venkatesh},
      year={2024},
      eprint={2409.04677},
      archivePrefix={arXiv},
      primaryClass={math.RT},
      url={https://arxiv.org/abs/2409.04677}, 
}

@misc{BF,
      title={Equivariant {Satake} category and {Kostant}-{Whittaker} reduction}, 
      author={Roman Bezrukavnikov and Michael Finkelberg},
      year={2008},
      eprint={0707.3799},
      archivePrefix={arXiv},
      primaryClass={math.RT},
      url={https://arxiv.org/abs/0707.3799}, 
}

@misc{liu2025higherperiodintegralsderivatives,
      title={Higher Period Integrals and Derivatives of {$L$}-functions}, 
      author={Shurui Liu and Zeyu Wang},
      year={2025},
      eprint={2504.00275},
      archivePrefix={arXiv},
      primaryClass={math.NT},
      url={https://arxiv.org/abs/2504.00275}, 
}

@misc{GR,
      title={}, 
      author={Dennis Gaitsgory and Sam Raskin},
}

@misc{FYZvolume,
      title={Arithmetic volumes of moduli stacks of {S}htukas}, 
      author={Tony Feng and Zhiwei Yun and Wei Zhang},
      year={2026},
      eprint={2601.18557},
      archivePrefix={arXiv},
      primaryClass={math.NT},
      url={https://arxiv.org/abs/2601.18557}, 
}

@article{yun2011integral,
  title={Integral homology of loop groups via {L}anglands dual groups},
  author={Yun, Zhiwei and Zhu, Xinwen},
  journal={Representation Theory of the American Mathematical Society},
  volume={15},
  number={9},
  pages={347--369},
  year={2011}
}

@manual{sagemath,
  Key          = {SageMath},
  Author       = {{The Sage Developers}},
  Title        = {{S}age{M}ath, the {S}age {M}athematics {S}oftware {S}ystem ({V}ersion 10.7)},
  note         = {{\tt https://www.sagemath.org}},
  Year         = {2025},
}

@misc{feng2026eigenweightsarithmetichirzebruchproportionality,
      title={Eigenweights for arithmetic {H}irzebruch Proportionality}, 
      author={Tony Feng},
      year={2026},
      eprint={2601.23245},
      archivePrefix={arXiv},
      primaryClass={math.RT},
      url={https://arxiv.org/abs/2601.23245}, 
}

@article{heinloth2010cohomology,
  title={The cohomology rings of moduli stacks of principal bundles over curves},
  author={Heinloth, Jochen and Schmitt, Alexander HW},
  journal={Documenta Mathematica},
  volume={15},
  pages={423--488},
  year={2010}
}

@misc{ginzburg1998loopgrassmanniancohomologyprincipal,
      title={Loop {G}rassmannian cohomology, the principal nilpotent and {K}ostant theorem}, 
      author={Victor Ginzburg},
      year={1998},
      eprint={math/9803141},
      archivePrefix={arXiv},
      primaryClass={math.AG},
      url={https://arxiv.org/abs/math/9803141}, 
}

@article{eigenweights,
  author = {Burak Aricioglu and Zeyu Wang},
  title = {A closed formula for eigenweights},
  year = {2026}
}
\end{document}